\documentclass{amsproc}

\usepackage[all,2cell]{xy} \UseAllTwocells \SilentMatrices
\usepackage{latexsym,amsfonts,amssymb}
\usepackage{amsmath,amsthm,amscd}
\usepackage{hyperref,psfrag}
\usepackage{diagcat}
\usepackage{color}
\usepackage{etoolbox}
\usepackage[dvips]{epsfig}
\usepackage{psfrag}

\usepackage{graphicx}

\usepackage{a4wide}
\usepackage{geometry}

\usepackage{epigraph,wrapfig}

\theoremstyle{definition}
\newtheorem{thm}{Theorem}[section]
\newtheorem{cor}[thm]{Corollary}

\newtheorem{lem}[thm]{Lemma}
\newtheorem{rem}[thm]{Remark}
\newtheorem{prop}[thm]{Proposition}
\newtheorem{defn}[thm]{Definition}
\newtheorem{example}[thm]{Example}

\numberwithin{equation}{section}

\usepackage{bbm}
\def\o{\otimes}

\def\N{{\mathbbm N}}

\def\Z{{\mathbbm Z}}
\def\Q{{\mathbbm Q}}

\def\F{{\mathbbm F}}

\def\1{{\mathbbm{1}}}

\newcommand{\END}{{\rm END}}

\def\dif{\partial}

\def\lra{{\longrightarrow}}
\def\gmod{{\mathrm{-gmod}}}   %% finitely-generated graded modules
\def\Id{\mathrm{Id}}
\def\mc{\mathcal}
\def\mf{\mathfrak}
\def\shuffle{\,\raise 1pt\hbox{$\scriptscriptstyle\cup{\mskip
               -4mu}\cup$}\,}
\newcommand{\refequal}[1]{\xy {\ar@{=}^{#1}
(-1,0)*{};(1,0)*{}};
\endxy}

\newcommand{\sym}{\mathrm{\Lambda}}

\newcommand{\pol}{\mathrm{Pol}}

\newcommand{\nh}{\mathrm{NH}}

\newcommand{\mH}{\mathrm{H}} %cohomology
\newcommand{\RHOM}{\mathbf{R}\mathrm{HOM}}

\newcommand{\cwbubble}[2]{
\begin{DGCpicture}
\DGCcoupon*(-0.4,-0.4)(.4,.4){ }
\DGCbubble(0,0){0.35}
\DGCdot*<{0.2,L}
\DGCdot*.{0.1,R}[r]{#2}
\DGCcoupon*(-.4,-.4)(.4,.4){\small{#1}}
\end{DGCpicture}
}

\newcommand{\ccwbubble}[2]{
\begin{DGCpicture}
\DGCcoupon*(-0.4,-0.4)(.4,.4){ }
\DGCbubble(0,0){0.35}
\DGCdot*>{0.2,L}
\DGCdot*.{0.1,R}[r]{#2}
\DGCcoupon*(-.4,-.4)(.4,.4){\small{#1}}
\end{DGCpicture}
}

\newcommand{\bigcwbubble}[2]{
\begin{DGCpicture}
\DGCcoupon*(-0.8,-0.8)(.8,.8){ }
\DGCbubble(0,0){0.5}
\DGCdot*<{0.25,L}
\DGCdot*.{0.25,R}[r]{#2}
\DGCcoupon*(-.6,-.6)(.6,.6){\small{#1}}
\end{DGCpicture}
}

\newcommand{\bigccwbubble}[2]{
\begin{DGCpicture}
\DGCcoupon*(-0.8,-0.8)(.8,.8){ }
\DGCbubble(0,0){0.5}
\DGCdot*>{0.25,L}
\DGCdot*.{0.25,R}[r]{#2}
\DGCcoupon*(-.6,-.6)(.6,.6){\small{#1}}
\end{DGCpicture}
}

\newcommand{\cwcapbubcup}[4]{
\begin{DGCpicture}
\DGCcoupon*(-0.3,-0.1)(1.3,2.1){ }
\ifstrequal{#4}{no}{}{
\DGCcoupon*(0,.8)(.3,1.4){#4}}
\ifstrequal{#1}{no}{}{
\DGCstrand(0,0)(1,0)/d/
\DGCdot*>{0.25,2}
\ifstrequal{#1}{$0$}{}{
\ifstrequal{#1}{$1$}{\DGCdot{0.25,1}}{
\DGCdot{0.25,1}[r]{\mbox{\scriptsize #1}}}}}
\ifstrequal{#3}{no}{}{
\DGCstrand/d/(0,2)(1,2)
\DGCdot*<{1.75,2}
\ifstrequal{#3}{$0$}{}{
\ifstrequal{#3}{$1$}{\DGCdot{1.75,1}}{
\DGCdot{1.75,1}[r]{\mbox{\scriptsize #3}}}}}
\ifstrequal{#2}{no}{}{
\DGCbubble(1,1){.3}
\DGCdot*>{1.2,L}
\DGCcoupon*(.65,.65)(1.35,1.35){\small{#2}}}
\end{DGCpicture}
}

\newcommand{\ccwcapbubcup}[4]{
\begin{DGCpicture}
\DGCcoupon*(-0.3,-0.1)(1.3,2.1){ }
\ifstrequal{#4}{no}{}{
\DGCcoupon*(0,.8)(.3,1.4){#4}}
\ifstrequal{#1}{no}{}{
\DGCstrand(0,0)(1,0)/d/
\DGCdot*<{0.25,2}
\ifstrequal{#1}{$0$}{}{
\ifstrequal{#1}{$1$}{\DGCdot{0.25,1}}{
\DGCdot{0.25,1}[r]{\mbox{\scriptsize #1}}}}}
\ifstrequal{#3}{no}{}{
\DGCstrand/d/(0,2)(1,2)
\DGCdot*>{1.75,2}
\ifstrequal{#3}{$0$}{}{
\ifstrequal{#3}{$1$}{\DGCdot{1.75,1}}{
\DGCdot{1.75,1}[r]{\mbox{\scriptsize #3}}}}}
\ifstrequal{#2}{no}{}{
\DGCbubble(1,1){.3}
\DGCdot*<{1.2,L}
\DGCcoupon*(.65,.65)(1.35,1.35){\small{#2}}}
\end{DGCpicture}
}

\newcommand{\RIII}[5]{
\begin{DGCpicture}[scale=0.85]
\DGCcoupon*(-0.3,-0.3)(2.3,2.3){}
\ifstrequal{#5}{no}{}{
\DGCcoupon*(2.1,.95)(2.3,1.15){#5}}
\DGCstrand(0,0)(2,2)
\DGCdot*>{2}
\ifstrequal{#4}{$0$}{}{
\ifstrequal{#4}{$1$}{\DGCdot{1.7}}{
\DGCdot{1.7}[r]{\mbox{\scriptsize #4}}}}
\DGCstrand(2,0)(0,2)
\DGCdot*>{2}
\ifstrequal{#2}{$0$}{}{
\ifstrequal{#2}{$1$}{\DGCdot{1.7}}{
\DGCdot{1.7}[r]{\mbox{\scriptsize #2}}}}
\ifstrequal{#1}{L}{\DGCstrand(1,0)(0,1)(1,2)}{\DGCstrand(1,0)(2,1)(1,2)}
\DGCdot*>{2}
\ifstrequal{#3}{$0$}{}{
\ifstrequal{#3}{$1$}{\DGCdot{1.7}}{
\DGCdot{1.7}[r]{\mbox{\scriptsize #3}}}}
\end{DGCpicture}
}

\newcommand{\twolines}[3]{
\begin{DGCpicture}
\DGCcoupon*(-.3,-.3)(1.3,1.3){}
\ifstrequal{#3}{no}{}{\DGCcoupon*(1.1,.6)(1.4,.9){#3}}
\DGCstrand(0,0)(0,1)
\DGCdot*>{1}
\ifstrequal{#1}{$0$}{}{
\ifstrequal{#1}{$1$}{\DGCdot{.4}}{
\DGCdot{.4}[r]{\mbox{\scriptsize #1}}}}
\DGCstrand(1,0)(1,1)
\DGCdot*>{1}
\ifstrequal{#2}{$0$}{}{
\ifstrequal{#2}{$1$}{\DGCdot{.4}}{
\DGCdot{.4}[r]{\mbox{\scriptsize #2}}}}
\end{DGCpicture}
}

\newcommand{\twolinesD}[3]{
\begin{DGCpicture}
\DGCcoupon*(-.3,-.3)(1.3,1.3){}
\ifstrequal{#3}{no}{}{\DGCcoupon*(-.4,.6)(-.1,.9){#3}}
\DGCstrand(0,0)(0,1)
\DGCdot*<{0}
\ifstrequal{#1}{$0$}{}{
\ifstrequal{#1}{$1$}{\DGCdot{.6}}{
\DGCdot{.6}[r]{\mbox{\scriptsize #1}}}}
\DGCstrand(1,0)(1,1)
\DGCdot*<{0}
\ifstrequal{#2}{$0$}{}{
\ifstrequal{#2}{$1$}{\DGCdot{.6}}{
\DGCdot{.6}[r]{\mbox{\scriptsize #2}}}}
\end{DGCpicture}
}

\newcommand{\crossing}[5]{
\begin{DGCpicture}
\DGCcoupon*(-.3,-.3)(1.3,1.3){}
\ifstrequal{#5}{no}{}{
\DGCcoupon*(1,.4)(1.3,.7){#5}}
\DGCstrand(0,0)(1,1)
\DGCdot*>{1}
\ifstrequal{#2}{$0$}{}{
\ifstrequal{#2}{$1$}{\DGCdot{.7}}{
\DGCdot{.7}[r]{\mbox{\scriptsize #2}}}}
\ifstrequal{#3}{$0$}{}{
\ifstrequal{#3}{$1$}{\DGCdot{.3}}{
\DGCdot{.3}[r]{\mbox{\scriptsize #3}}}}
\DGCstrand(1,0)(0,1)
\DGCdot*>{1}
\ifstrequal{#1}{$0$}{}{
\ifstrequal{#1}{$1$}{\DGCdot{.7}}{
\DGCdot{.7}[r]{\mbox{\scriptsize #1}}}}
\ifstrequal{#4}{$0$}{}{
\ifstrequal{#4}{$1$}{\DGCdot{.3}}{
\DGCdot{.3}[r]{\mbox{\scriptsize #4}}}}
\end{DGCpicture}
}

\newcommand{\crossingD}[5]{
\begin{DGCpicture}
\DGCcoupon*(-.3,-.3)(1.3,1.3){}
\ifstrequal{#5}{no}{}{\DGCcoupon*(-.3,.4)(0,.7){#5}}
\DGCstrand(0,0)(1,1)
\DGCdot*<{0}
\ifstrequal{#2}{$0$}{}{
\ifstrequal{#2}{$1$}{\DGCdot{.3}}{
\DGCdot{.3}[r]{\mbox{\scriptsize #2}}}}
\ifstrequal{#3}{$0$}{}{
\ifstrequal{#3}{$1$}{\DGCdot{.7}}{
\DGCdot{.7}[r]{\mbox{\scriptsize #3}}}}
\DGCstrand(1,0)(0,1)
\DGCdot*<{0}
\ifstrequal{#1}{$0$}{}{
\ifstrequal{#1}{$1$}{\DGCdot{.3}}{
\DGCdot{.3}[r]{\mbox{\scriptsize #1}}}}
\ifstrequal{#4}{$0$}{}{
\ifstrequal{#4}{$1$}{\DGCdot{.7}}{
\DGCdot{.7}[r]{\mbox{\scriptsize #4}}}}
\end{DGCpicture}
}

\newcommand{\oneline}[2]{
\begin{DGCpicture}
\DGCcoupon*(-.3,-.1)(0.3,2.1){}
\DGCstrand(0,0)(0,2)
\DGCdot*>{2}
\ifstrequal{#2}{no}{}{\DGCcoupon*(.1,1.4)(.4,1.7){#2}}
\ifstrequal{#1}{$0$}{}{
\ifstrequal{#1}{$1$}{\DGCdot{1}}{
\DGCdot{1}[r]{\mbox{\scriptsize #1}}}}
\end{DGCpicture}
}

\newcommand{\onelineD}[2]{
\begin{DGCpicture}
\DGCcoupon*(-.3,-.1)(0.3,2.1){}
\DGCstrand(0,0)(0,2)
\DGCdot*<{0}
\ifstrequal{#2}{no}{}{\DGCcoupon*(-.4,1.4)(-.1,1.7){#2}}
\ifstrequal{#1}{$0$}{}{
\ifstrequal{#1}{$1$}{\DGCdot{1}}{
\DGCdot{1}[r]{\mbox{\scriptsize #1}}}}
\end{DGCpicture}
}

\newcommand{\onelineshort}[2]{
\begin{DGCpicture}
\DGCcoupon*(-.3,-.1)(0.3,1.1){}
\DGCstrand(0,0)(0,1)
\DGCdot*>{1}
\ifstrequal{#2}{no}{}{\DGCcoupon*(.1,.7)(.4,1){#2}}
\ifstrequal{#1}{$0$}{}{
\ifstrequal{#1}{$1$}{\DGCdot{.5}}{
\DGCdot{.5}[r]{\mbox{\scriptsize #1}}}}
\end{DGCpicture}
}

\newcommand{\onelineDshort}[2]{
\begin{DGCpicture}
\DGCcoupon*(-.3,-.1)(0.3,1.1){}
\DGCstrand(0,0)(0,1)
\DGCdot*<{0}
\ifstrequal{#2}{no}{}{\DGCcoupon*(-.4,.7)(-.1,1){#2}}
\ifstrequal{#1}{$0$}{}{
\ifstrequal{#1}{$1$}{\DGCdot{.5}}{
\DGCdot{.5}[r]{\mbox{\scriptsize #1}}}}
\end{DGCpicture}
}

\newcommand{\curl}[5]{
\begin{DGCpicture}
\ifstrequal{#1}{L}{
\DGCcoupon*(-2,-.8)(.3,1.8){}
\DGCstrand(0,-.5)(0,.25)/u/(-1.5,.5)/d/(0,.75)/u/(0,1.5)
\ifstrequal{#4}{no}{}{\DGCcoupon*(-1.4,0)(-.4,1){\small{#4}}}
\ifstrequal{#5}{$0$}{}{\ifstrequal{#5}{$1$}{\DGCdot{.5,4}}{\DGCdot{.5,4}[l]{\mbox{\scriptsize #5}}}}
\ifstrequal{#3}{no}{}{\DGCcoupon*(-1.2,1.1)(-.9,1.4){#3}}
}{
\DGCcoupon*(-.3,-.8)(2,1.8){}
\DGCstrand(0,-.5)(0,.25)/u/(1.5,.5)/d/(0,.75)/u/(0,1.5)
\ifstrequal{#4}{no}{}{\DGCcoupon*(.4,0)(1.4,1){\small{#4}}}
\ifstrequal{#5}{$0$}{}{\ifstrequal{#5}{$1$}{\DGCdot{.5,4}}{\DGCdot{.5,4}[r]{\mbox{\scriptsize #5}}}}
\ifstrequal{#3}{no}{}{\DGCcoupon*(.9,1.1)(1.2,1.4){#3}}
}
\ifstrequal{#2}{D}{\DGCdot*<{-0.25} \DGCdot*<{1.25} \DGCdot*<{.75,2}}{\DGCdot*>{-0.25} \DGCdot*>{1.25} \DGCdot*>{.75,2}}
\end{DGCpicture}
}

\newcommand{\cappy}[4]{
\begin{DGCpicture}
\DGCcoupon*(-.3,-.1)(1.3,.8){}
\DGCstrand(0,0)(1,0)/d/
\ifstrequal{#1}{CCW}{\DGCdot*<{0,2}}{\DGCdot*>{0,1}}
\ifstrequal{#2}{$0$}{}{\ifstrequal{#2}{$1$}{\DGCdot{.3}}{\DGCdot{.3}[r]{\mbox{\scriptsize #2}}}}
\ifstrequal{#3}{no}{}{
\DGCbubble(1.7,.4){0.3}
\ifstrequal{#3}{CCW}{\DGCdot*>{.7,L}}{\DGCdot*<{.7,L}}
\DGCcoupon*(1.35,0)(2.05,.8){\small{$1$}}}
\ifstrequal{#4}{no}{}{\DGCcoupon*(.8,.5)(1.2,.8){#4}}
\end{DGCpicture}
}

\newcommand{\cuppy}[4]{
\begin{DGCpicture}
\DGCcoupon*(-.3,0.2)(1.3,1.1){}
\DGCstrand(0,1)/d/(1,1)/u/
\ifstrequal{#1}{CCW}{\DGCdot*>{1,1}}{\DGCdot*<{1,2}}
\ifstrequal{#2}{$0$}{}{\ifstrequal{#2}{$1$}{\DGCdot{.7}}{\DGCdot{.7}[r]{\mbox{\scriptsize #2}}}}
\ifstrequal{#3}{no}{}{
\DGCbubble(1.7,.6){0.3}
\ifstrequal{#3}{CCW}{\DGCdot*>{.9,L}}{\DGCdot*<{.9,L}}
\DGCcoupon*(1.35,.2)(2.05,1){\small{$1$}}}
\ifstrequal{#4}{no}{}{\DGCcoupon*(.8,.2)(1.2,.5){#4}}
\end{DGCpicture}
}

\newcommand{\bottomcurl}[5]{
\begin{DGCpicture}
\DGCcoupon*(-.3,-.1)(1.3,1.6){}
\DGCstrand(0,0)(1,1)/u/(0,1)/d/(1,0)/d/
\ifstrequal{#1}{CW}{\DGCdot*>{1.3}}{\DGCdot*<{1.3}}
\ifstrequal{#2}{no}{}{\DGCcoupon*(0,0.65)(1,1.3){\mbox{\scriptsize #2}}}
\ifstrequal{#3}{yes}{\DGCdot{.3,2}}{}
\ifstrequal{#4}{no}{}{
\DGCbubble(1.5,0.6){.3}
\ifstrequal{#4}{CCW}{\DGCdot*>{.8,L}}{\DGCdot*<{.8,L}}
\DGCcoupon*(1.3,0.4)(1.7,0.8){\small{$1$}}}
\ifstrequal{#5}{no}{}{\DGCcoupon*(1.1,0.9)(1.5,1.5){#5}}
\end{DGCpicture}
}

\newcommand{\crossingR}[4]{
\begin{DGCpicture}
\DGCcoupon*(-.3,-.3)(1.3,1.3){}
\DGCstrand(0,0)(1,1)
\DGCdot*>{1}
\ifstrequal{#1}{no}{}{\DGCdot{.65}}
\DGCstrand(1,0)(0,1)
\DGCdot*<{0}
\ifstrequal{#2}{no}{}{\DGCdot{.65}}
\ifstrequal{#3}{no}{}{
\DGCbubble(1.2,.5){0.3}
\DGCdot*<{.7,R}
\DGCcoupon*(0.9,0.2)(1.5,.8){\small{$1$}}}
\ifstrequal{#4}{no}{}{
\DGCcoupon*(-.3,.2)(.3,.8){#4}}
\end{DGCpicture}
}

\newcommand{\crossingL}[4]{
\begin{DGCpicture}
\DGCcoupon*(-.3,-.3)(1.3,1.3){}
\DGCstrand(0,0)(1,1)
\DGCdot*<{0}
\ifstrequal{#1}{no}{}{\DGCdot{.35}}
\DGCstrand(1,0)(0,1)
\DGCdot*>{1}
\ifstrequal{#2}{no}{}{\DGCdot{.35}}
\ifstrequal{#3}{no}{}{
\DGCbubble(-.2,.5){0.3}
\DGCdot*<{.7,L}
\DGCcoupon*(-.5,0.2)(.1,.8){\small{$1$}}}
\ifstrequal{#4}{no}{}{
\DGCcoupon*(.7,.2)(1.3,.8){#4}}
\end{DGCpicture}
}

\begin{document}

\title{Categorification at prime roots of unity and hopfological finiteness}

\author{You Qi}
\address{Department of Mathematics, Yale University, New Haven, CT 06511, USA}
%\curraddr{Department of Mathematics, Yale University, New Haven, CT 06511, USA}
\email{you.qi@yale.edu}
\thanks{The authors are very grateful to Mikhail Khovanov for his support and encouragement, as well as sharing with us his insights on the subject.
The authors would also like to thank Catharina Stroppel for helpful discussions about categorifying Jones-Wenzl projectors.}

\author{Joshua Sussan}
\address{Department of Mathematics, CUNY Medgar Evers, Brooklyn, NY, 11225, USA}
\email{jsussan@mec.cuny.edu}
\thanks{J.S. was supported by NSF grant DMS-1407394, PSC-CUNY Award 67144-00 45, and an Alfred P. Sloan
Foundation CUNY Junior Faculty Award.}

%    General info
\date{\today}

\dedicatory{Dedicated to Christian Blanchet on the occasion of his sixtieth birthday}


\begin{abstract}
We survey some recent results in hopfological algebra and the program of categorification at prime roots of unity. A categorical Jones-Wenzl projector at prime roots of unity is studied, and it is shown that this projector is hopfologically finite in special characteristics, while generically it is infinite.
\end{abstract}

\maketitle

\section{Introduction}
The Witten-Reshetikhin-Turaev (WRT) invariants \cite{Witten, RT} of links and tangles can be constructed from quantum groups at generic values of $q$. For the purposes of this note, we will only focus on the $\mathfrak{sl}_2$-case. When $q$ equals a root of unity, these invariants specialized at $q$ naturally arise, from a quantum physical point of view, as invariants assigned to certain decorated $3$-spheres using the $3$d-WRT topological quantum field theory (TQFT) $Z_n$. Here the 3d TQFT $Z_n$, parameterized by the order of $q$ (which is related to the level of the TQFT), admits representation theoretical explanations from quantum groups at roots of unity \cite{RT}.

The goal of the categorification program envisioned by Crane and Frenkel~\cite{CF} was to lift the WRT $3$-manifold invariant $Z_n(M^3)$, whose values live in the ring of cyclotomic integers $\mathbb{Z}[\zeta_n]$ (see, for instance, \cite{BeChLe}), to a homological invariant.
If this hypothetical homological structure were truly a functorial theory, then one should obtain invariants of $4$-manifolds with boundaries, which are viewed as cobordisms between $3$-manifolds. This would in turn lift $Z_n$ to a $4$-dimensional theory, as conjectured in \cite{CF}.

Strong evidence of the plausibility of Crane and Frenkel's proposal is provided by Khovanov homology \cite{KhJones}, which assigns a functorial homological invariant to any (framed) tangle. Since this ground-breaking work, there has been significant progress towards categorifying link invariants at generic $q$-values. We refer the reader to the work of Webster \cite{Webcombined} (and the references therein) for more historical accounts, as well as a striking general construction of homological WRT tangle invariants.

In 2004, Khovanov proposed the first steps towards categorification at roots of unity ~\cite{Hopforoots}.  He introduced the subject of \emph{hopfological algebra}, which combines theory of Hopf algebras with homological algebra.  The key observation of \cite{Hopforoots} is that, to categorify the ring of cyclotomic integers at a prime root of unity, one could utilize Mayer's homotopy category of $p$-complexes $\Bbbk[\dif]/(\dif^p)\!-\!\underline{\mathrm{gmod}}$~\cite{Mayer1,Mayer2} over a field of characteristic $p>0$:
\begin{equation}\label{eqn-categorifying-Op}
K_0(\Bbbk[\dif]/(\dif^p)\!-\!\underline{\mathrm{gmod}})\cong \Z[\zeta_p].
\end{equation}
Mayer's $p$-complexes can be understood as a special case of hopfological algebra, in the sense that it is the hopfological theory of the ground field $\Bbbk$ associated with the graded Hopf algebra $H=\Bbbk[\dif]/(\dif^p)$, where $\mathrm{char}(\Bbbk)=p$ and $\mathrm{deg}(\dif)=1$. A characteristic-zero analogue of Mayer's construction, known as $n$-complexes, has been considered by Kapranov \cite{Ka} and Sakaria \cite{SarNotes}. This subject is, roughly speaking, a hopfological theory attached to the Taft algebra $H_n$ at an $n$th root of unity. However, due to the lack of a braiding functor on $H_n\!-\!\mathrm{mod}$, it is relatively difficult to implement this theory in the categorification program. It remains an interesting open problem to devise a monoidal categorification of a general cyclotomic integer ring\footnote{See the recent preprint \cite{Mirmohades} for an interesting new case.}, and apply such a construction in categorification of the WRT 3d TQFT.

In order to apply the ``categorical cyclotomic integers" $\Bbbk[\dif]/(\dif^p)\!-\!\underline{\mathrm{gmod}}$ to categorification, one needs a good supply of algebra objects in the monoidal category $\Bbbk[\dif]/(\dif^p)\!-\!\underline{\mathrm{gmod}}$. Such algebra objects arise naturally from $p$-differential graded ($p$-DG) algebras, an analogue of the usual DG-algebra over graded (super)vector spaces. Some basic technical machinery of hopfological algebra, and in particular, $p$-DG algebra, has been developed in~\cite{QYHopf}. The theory parallels that of the usual DG theory, and in fact the usual DG theory is a special case of hopfological algebra for the Hopf superalgebra $\Bbbk[d]/(d^2)$. For instance, one has the notion of the abelian category\footnote{These notions will be recalled in Section~\ref{sec-zig-zag}.} of $p$-DG modules over a $p$-DG algebra $A$, its homotopy category of $p$-DG modules $\mathcal{H}(A)$, and the derived category $\mathcal{D}(A)$ obtained from $\mathcal{H}(A)$ under localization at quasi-isomorphisms. Both categories $\mathcal{H}(A)$ and $\mathcal{D}(A)$ afford categorical module structures over $H\!-\!\underline{\mathrm{gmod}}$ under tensor product by $p$-complexes. Taking Grothendieck groups in the appropriate sense, the decategorification realizes $K_0(\mathcal{D}(A))$ (or $K_0(\mathcal{H}(A))$, but the former is more relevant in our story) as modules over $\Z[\zeta_p]$:
\begin{equation}\label{eqn-categorical-Op-mod}
\begin{gathered}
\xymatrix{
\Bbbk[\dif]/(\dif^p)\!-\!\underline{\mathrm{gmod}} \times \mathcal{D}(A) \ar[rr]^-{\otimes} \ar@2{->}[dd]^{K_0} && \mathcal{D}(A)\ar@2{->}[dd]^{K_0}\\
&& \\
K_0(\Bbbk[\dif]/(\dif^p)\!-\!\underline{\mathrm{gmod}}) \times K_0(\mc{D}(A)) \ar[rr]^-{\times}  && K_0(\mathcal{D}(A)).
}
\end{gathered}
\end{equation}
One can also easily check that the usual associativity and distributivity axioms of a module over $\Z[\zeta_p]$ are lifted to the corresponding properties of direct sums and tensor products in the triangulated categories.

In this note we will sketch out what interesting $p$-DG algebras have been investigated in earlier works. A natural way to find nontrivial $p$-DG algebras that may be relevant to quantum topology is to combine hopfological algebra of $p$-differentials with the machineries that have successfully lifted integral quantum structures (for generic $q$). It appears to us that, so far, building on the beautiful works of \cite{KL1,KL2,KL3,Lau1,Rou2,Webcombined} etc., working along this direction seems promising.  A summary of what is already known will be carried out in the next few sections, and at the end we will point out one new interesting feature of a hopfological finiteness property and its dependence on the characteristic of the ground field.

In Section~\ref{sec-small-qgroup} we review quantum $\mathfrak{sl}_2$ at roots of unity and related structures. The results are by now classical, and are included to fix notation for the rest of the paper.

Starting from Section~\ref{sec-zig-zag}, we recall the definition of a $p$-DG algebra and some basic hopfological properties. Following \cite{QiSussan}, we illustrate the definitions with a simple example $A_n^!$, whose derived category carries a categorical braid group action. Via this example, we hope the reader is convinced that hopfological algebra is indeed connected to quantum topology at roots of unity.

Then we proceed to more representation theoretical aspects of the story. In Section \ref{sec-cat-sl(2)}, we review the $p$-DG nilHecke algebra defined by Khovanov and Qi ~\cite{KQ} used to categorify the lower half of quantum $\mathfrak{sl}_2$ at prime roots of unity. This is further extended by Elias and Qi~\cite{EQ1} to categorify the full small quantum group, following Lauda's \cite{Lau1} work on the generic case. A categorification of the Beilinson-Lusztig-MacPherson (BLM) form of quantum $\mathfrak{sl}_2$ at prime roots of unity has been constructed in \cite{EQ2} by equipping the thick calculus of Khovanov-Lauda-Mackaay-Sto\v{s}i\'{c} \cite{KLMS} with a compatible $p$-differential.

A categorified Jones-Wenzl projector is constructed in Section \ref{sec-jones-wenzl}.  When equipped with $p$-differentials, we find some interesting new phenomenon. In certain characteristics, this functor is quasi-isomorphic to a finite complex, although for almost all primes it is an infinite complex unbounded in one direction. This suggests to us that the homology of the colored unknot may be finite only in special prime characteristics. We hope to investigate this phenomenon more carefully in upcoming works.

Finally, we should mention a recent perspective on variants of WRT TQFTs. The original WRT TQFTs build upon semisimplification of certain tensor subcategories of representation theory of quantum groups at roots of unity. A very recent construction of Blanchet, Costantino, Geer and Patureau-Mirand~\cite{BCGM} utilizes only projective objects in the representation category to define a non semisimple TQFT. It is an exciting problem to categorify their story at a prime root of unity.

\section{The small quantum group}\label{sec-small-qgroup}
\subsection{The small quantum group.}
Let $ \zeta_{2l} $ be a primitive $2l$th root of unity.  The quantum group $ u_{\Q[\zeta_{2l}]}(\mathfrak{sl}_2) $, which we will denote simply by $ u_{\Q[\zeta_{2l}]} $, is the $\Q[\zeta_{2l}]$-algebra generated by $ E, F, K^{\pm 1} $ subject to relations:
\begin{enumerate}
\item $ KK^{-1} = K^{-1}K=1$,
\item $ K^{\pm 1}E = \zeta_{2l}^{\pm 2} EK^{\pm 1} $, \quad $ K^{\pm 1} F = \zeta_{2l}^{\mp 2} FK^{\pm 1} $,
\item $ EF-FE = \frac{K-K^{-1}}{\zeta_{2l}-\zeta_{2l}^{-1}} $,
\item $ E^l = F^l = 0 $.
\end{enumerate}

The quantum group is a Hopf algebra whose comultiplication map 
$$ \Delta \colon u_{\Q[\zeta_{2l}]} \longrightarrow u_{\Q[\zeta_{2l}]} \otimes_{\Q[\zeta_{2l}]} u_{\Q[\zeta_{2l}]} $$
is given on generators by
\begin{equation*}
\Delta(E) = E \otimes 1 + K \otimes E, \quad\quad \Delta(F) = 1 \otimes F + F \otimes K^{-1}, \quad\quad \Delta(K^{\pm 1}) = K^{\pm 1} \otimes K^{\pm 1}.
\end{equation*}

For the purpose of categorification, it is more convenient to use the idempotented quantum group $ \dot{u}_{\Q[\zeta_{2l}]}(\mathfrak{sl}_2) $, which we will denote simply by $ \dot{u}_{\Q[\zeta_{2l}]} $. It is a non-unital $\Q[\zeta_{2l}]$-algebra generated by $ E, F $ and idempotents $1_\lambda$ ($\lambda \in \Z$), subject to relations:
\begin{enumerate}
\item $ 1_\lambda 1_\mu=\delta_{\lambda, \mu}1_\lambda$,
\item $ E1_\lambda=1_{\lambda+2}E $, \quad $ F1_{\lambda+2} = 1_{\lambda} F $,
\item $ EF1_{\lambda}-FE1_{\lambda} = [\lambda]1_\lambda $,
\item $ E^l = F^l = 0 $.
\end{enumerate}
Here $ [\lambda] =\sum_{i=0}^{|\lambda|-1} \zeta_{2l}^{1-|\lambda|+2i} $ is the quantum integer specialized at $\zeta_{2l}$.

The quantum group $\dot{u}_{\Q[\zeta_{2l}]}$ has an integral lattice subalgebra which we now recall. For any integer $n\in \{0,1,\dots, l-1\}$, let $ E^{(n)} = \frac{E^n}{[n]!} $, and $ F^{(n)} = \frac{F^n}{[n]!} $.
The elements $ E^{(n)}$, $F^{(n)}$ ($0\leq n \leq l-1$), and $ 1_{\lambda}$ ($\lambda\in \Z$) generate an algebra over the ring of cyclotomic integers $ \mathcal{O}_{2l} = \mathbb{Z}[\zeta_{2l}] $. Denote this integral form by $ \dot{u}_{\zeta_{2l}} $.

Now let $ l=p $ be prime.  Introduce the auxiliary ring
\begin{equation}\label{eqn-aux-ring}
\mathbb{O}_p = \mathbb{Z}[q]/(\Psi_p(q^2)),
\end{equation}  where $ \Psi_p(q) $ is the $p$-th cyclotomic polynomial.
We can define in a similar fashion an integral form  $ {u}_{\mathbb{O}_p} $ and its dotted version $ \dot{u}_{\mathbb{O}_p} $ for the small quantum $\mf{sl}_2$ over $ \mathbb{O}_p $.  For more details see \cite[Section 3.3]{KQ}. Let the lower half of $ u_{\mathbb{O}_p} $ be the subalgebra generated by the $ F^{(n)} $ ($0\leq n \leq p-1$) and denote it by $ u_{\mathbb{O}_p}^- $. Likewise, write the upper half as $u_{\mathbb{O}_p}^+$.

Let $ {V}_n $ be the unique (up to isomorphism) irreducible module for $ \dot{u}_{\mathbb{O}_p} $ of rank $ n+1 $.  It has a basis $ \lbrace v_0, v_1, \ldots, v_{n} \rbrace $ such that
\begin{equation}
\label{irreddef}
 F 1_\lambda v_i =\delta_{\lambda, -n+2i}[n-i+1] v_{i-1}\quad\quad E1_{\lambda} v_i = \delta_{\lambda, -n+2i}[i+1] v_{i+1}.
\end{equation}

On $ V_1^{\otimes n} $ there is an action of $ {u}_{\mathbb{O}_p} $ given via the comultiplication map $ \Delta $.  There is also a commuting action of the braid group that factors through the Hecke algebra.  We mostly consider the second lowest weight space which we denote by
$ V_1^{\otimes n}[-n+2] $.  This is spanned by vectors $ v_{i_1} \otimes \cdots \otimes v_{i_n} $ where $ i_r \in \lbrace 0, 1 \rbrace $ and exactly one $ i_r$ is $ 1 $.
More generally $ V_1^{\otimes n}[-n+2k] $ is spanned by
vectors $ v_{i_1} \otimes \cdots \otimes v_{i_n} $ where $ i_r \in \lbrace 0, 1 \rbrace $ and $k$ of the $ i_r$ are $ 1 $.
It is convenient to use another basis of $ V_1^{\otimes n}[-n+2] $ which is spanned by vectors $ l_r$ where
%$ v_{i_1} \heartsuit \cdots \heartsuit v_{i_n} $ where $ i_r \in \lbrace 0, 1 \rbrace $ and exactly one $ i_r$ is $ 1 $.
\begin{equation*}
l_r =
\begin{cases}
v_{0}^{\otimes (r-1)} \otimes (v_1 \otimes v_0-qv_0\o v_1) \otimes v_{0}^{\otimes (n-r-1)}  &  r \neq n \\
v_{0}^{\otimes (n-1)} \otimes v_1 & r=n.
\end{cases}
\end{equation*}

\subsection{The Temperley-Lieb algebra.}
The \emph{Temperley-Lieb algebra $TL_n$} is the $\mathbb{O}_p$-algebra generated by elements $U_i$ for $i=1,\ldots,n-1$ subject to relations
\begin{enumerate}
\item $U_i^2=-(q+q^{-1})U_i$,
\item $U_i U_j= U_j U_i$ for $|i-j|>1$,
\item $ U_i U_{j} U_i = U_i $ for $ |i-j|=1$.
\end{enumerate}
Define an operator $\tilde{U}$ on $ V_1^{\otimes 2}$ by
\begin{align*}
\tilde{U}(v_0 \otimes v_0) &= 0 \\
\tilde{U}(v_1 \otimes v_1) &= 0 \\
\tilde{U}(v_0 \otimes v_1) &= v_1 \otimes v_0 - q v_0 \otimes v_1 \\
\tilde{U}(v_1 \otimes v_0) &= -q^{-1}(v_1 \otimes v_0 - q v_0 \otimes v_1).
\end{align*}
Then there is an action of the Temperley-Lieb algebra on $V_1^{\otimes n}$ where the generator $ {U}_i \colon V_1^{\otimes n} \rightarrow V_1^{\otimes n} $ acts by
\begin{equation*}
U_i = \Id^{\otimes (i-1)} \otimes \tilde{U} \otimes \Id^{\otimes (n-i-1)}.
\end{equation*}

In terms of the basis $ \{l_1, \ldots, l_n \} $ the action of the Temperley-Lieb algebra is given by
\begin{equation}
\label{templiebdualcan1}
U_i(l_j) =
\begin{cases}
-(q+q^{-1}) l_i & \text{ if } i=j \\
l_i & \text{ if } |i-j|=1 \\
0 & \text{ if } |i-j|>1.
\end{cases}
\end{equation}

\subsection{The braid group.}
The braid group $B_n$ is generated by elements $ t_i $ for $i =1, \ldots, n-1 $ subject to the relations that
\begin{enumerate}
\item $t_i t_j= t_j t_i$ if $|i-j|>1$,
\item $ t_i t_{i +1} t_i = t_{i+1} t_i t_{i+1} $ for $ i=1,\ldots,n-2$.
\end{enumerate}
For $i=1,\ldots, n-1$, it will be convenient to introduce the inverse of $t_i$ and denote it by $t_i^\prime$.

Define an operator $ \tilde{t} $ on $ V_1^{\otimes 2} $ by
\begin{align*}
\tilde{t}(v_0 \otimes v_0) &= v_0 \otimes v_0 \\
\tilde{t}(v_1 \otimes v_1) &= v_1 \otimes v_1 \\
\tilde{t}(v_0 \otimes v_1) &= q v_1 \otimes v_0 +(1-q^2) v_0 \otimes v_1 \\
\tilde{t}(v_1 \otimes v_0) &= q v_0 \otimes v_1.
\end{align*}
Then there is an action of the braid group generator $ {t}_i \colon V_1^{\otimes n} \rightarrow V_1^{\otimes n} $ given by
\begin{equation*}
t_i = \Id^{\otimes (i-1)} \otimes \tilde{t} \otimes \Id^{\otimes (n-i-1)}.
\end{equation*}
Define an operator $ \tilde{t}' $ on $ V_1^{\otimes 2} $ by
\begin{align*}
\tilde{t}'(v_0 \otimes v_0) &= v_0 \otimes v_0 \\
\tilde{t}'(v_1 \otimes v_1) &= v_1 \otimes v_1 \\
\tilde{t}'(v_0 \otimes v_1) &= q^{-1} v_1 \otimes v_0 \\
\tilde{t}'(v_1 \otimes v_0) &=  (1-q^{-2})v_1 \otimes v_0 + q^{-1} v_0 \otimes v_1.
\end{align*}
Then there is an action of the braid group generator $ {t}_i' \colon V_1^{\otimes n} \rightarrow V_1^{\otimes n} $ given by
\begin{equation*}
t_i' = \Id^{\otimes (i-1)} \otimes \tilde{t}' \otimes \Id^{\otimes (n-i-1)}.
\end{equation*}
It is easy to see that these operators preserve each weight space of $V_1^{\otimes n} $ and it is straightforward to check that they do indeed satisfy the braid relations.
The action of $ B_n $ on $ V_1^{\otimes n}[-n+2] $ is the Burau representation.

In terms of the basis $ \{l_1, \ldots, l_n \} $ the action of the braid group is given by
\begin{equation}
\label{twistingdualcan1}
t_i(l_j) =
\begin{cases}
-q^2 l_i & \text{ if } i=j \\
q l_i + l_j & \text{ if } |i-j|=1 \\
l_j & \text{ if } |i-j|>1
\end{cases}
\end{equation}
\begin{equation}
\label{twistingdualcan2}
t_i'(l_j) =
\begin{cases}
-q^{-2} l_i & \text{ if } i=j \\
q^{-1} l_i + l_j & \text{ if } |i-j|=1 \\
l_j & \text{ if } |i-j|>1.
\end{cases}
\end{equation}

\subsection{The Jones-Wenzl projector.}
The Jones-Wenzl projector is the unique linear map
\begin{equation*}
p_n \colon V_1^{\otimes n} \rightarrow V_n \rightarrow V_1^{\otimes n}
\end{equation*}
which commutes with the action of the quantum group and is idempotent.
There are many formulas for this map, usually given in terms of Temperley-Lieb generators.  When restricting to weight spaces we could write this map in terms of generators of the quantum group
\begin{equation*}
\frac{1}{{n \brack {k}}} E^{(k)} F^{(k)} \colon V_1^{\otimes n}[-n+2k] \rightarrow V_1^{\otimes n}[-n+2k].
\end{equation*}
For full details see \cite{StroppelSussanColorJ}.

\begin{prop}
The operator $\frac{1}{{n \brack {k}}} E^{(k)} F^{(k)} \colon V_1^{\otimes n} [-n+2k] \rightarrow V_1^{\otimes n} [-n+2k] $ is equal to the restriction of the Jones-Wenzl projector $p_n[-n+2k]$.
\end{prop}

\begin{proof}
One checks that $\frac{1}{{n \brack {k}}} E^{(k)} F^{(k)}$ is idempotent, commutes with $E$ and $F$ and is annihilated by the elements ${U}_i$ of the Temperley-Lieb algebra.  For more details see \cite{StroppelSussanColorJ}.
\end{proof}

\section{A categorical braid group action at a prime root of unity}\label{sec-zig-zag}
\subsection{Elements of $p$-DG algebras.} As a matter of notation for the rest of the paper, the undecorated tensor product symbol $\otimes$ will always denote tensor product over the ground field $\Bbbk$. We first recall some basic notions.
\begin{defn}\label{def-p-DGA}Let $\Bbbk$ be a field of positive characteristic $p$. A $p$-DG algebra $A$ over $\Bbbk$ is a $\Z$-graded $\Bbbk$-algebra equipped with a degree-two\footnote{In general one should define the degree of $\dif_A$ to be one. We adopt this degree only to match earlier grading conventions in categorification. One may adjust the gradings of the algebras we consider so as to make the degree of $\dif_A$ to be one, but we choose not to do so.} endomorphism $\dif_A$, such that, for any elements $a,b\in A$, we have
\[
\dif_A^p(a)=0, \quad \quad \dif_A(ab)=\dif_A(a)b+a\dif_A(b).
\]
\end{defn}

Compared with the usual DG case, the lack of the usual sign in the second equation above is because of the fact that the Hopf algebra $\Bbbk[\dif]/(\dif^p)$ is a genuine Hopf algebra, not a Hopf super-algebra.

As in the DG case, one has the notion of left and right $p$-DG modules.

\begin{defn}\label{def-p-DG-module}Let $(A,\dif_A)$ be a $p$-DG algebra. A left $p$-DG module $(M,\dif_M)$ is a $\Z$-graded $A$-module endowed with a degree-two endomorphism $\dif_M$, such that, for any elements $a\in A$ and $m\in M$, we have
\[
\dif_M^p(m)=0, \quad \quad \dif_M(am)=\dif_A(a)m+a\dif_M(m).
\]
Similarly, one has the notion of a right $p$-DG module.
\end{defn}

It is readily checked that the category of left (right) $p$-DG modules, denoted $A_\dif\!-\!\mathrm{mod}$ ($A^{op}_\dif\!-\!\mathrm{mod}$),  is abelian, with morphisms grading preserving $A$-module maps that also commute with differentials. When no confusion can be caused, we will drop all subscripts in differentials.

\begin{defn} \label{def-null-homotopy}
Let $M$ and $N$ be two $p$-DG modules. A morphism $f:M\lra N$ in $A_\dif\!-\!\mathrm{mod}$ is called \emph{null-homotopic} if there is an $A$-module map $h$ of degree $2-2p$ such that
\[
f=\sum_{i=0}^{p-1}\dif_N^{i}\circ h \circ \dif_M^{p-1-i}.
\]
\end{defn}

It is an easy exercise to check that null-homotopic morphisms form an ideal in $A_\dif\!-\!\mathrm{mod}$. The resulting quotient category, denoted $\mathcal{H}(A)$, is called the \emph{homotopy category} of left $p$-DG modules over $A$, and it is a triangulated category.

The simplest $p$-DG algebra is the ground field equipped with the trivial differential, whose homotopy category is denoted $\Bbbk[\dif]/(\dif^p)\!-\!\underline{\mathrm{gmod}}$\footnote{This is usually known as the graded stable category of $\Bbbk[\dif]/(\dif^p)$. }. In general, given any $p$-DG algebra $A$, one has a forgetful functor
\begin{equation}
\mathrm{For}:\mathcal{H}(A)\lra \Bbbk[\dif]/(\dif^p)\!-\!\underline{\mathrm{gmod}}
\end{equation}
by remembering only the underlying $p$-complex structure up to homotopy of any $p$-DG module over $A$. A morphism between two $p$-DG modules $f:M\lra N$ (or its image in the homotopy category) is called a \emph{quasi-isomorphism} if $\mathrm{For}(f)$ is an isomorphism in $\Bbbk[\dif]/(\dif^p)\!-\!\underline{\mathrm{gmod}}$. Denoting the class of quasi-isomorphisms in $\mathcal{H}$ by $\mathcal{Q}$, we define the $p$-DG derived category of $A$ to be
\begin{equation}\label{eqn-derived-cat}
\mathcal{D}(A):=\mathcal{H}(A)[\mathcal{Q}^{-1}],
\end{equation}
the localization of $\mathcal{H}(A)$ at quasi-isomorphisms. By construction, $\mathcal{D}(A)$ is triangulated.

Many constructions in the usual homological algebra of DG-algebras translate over into the $p$-DG context without any trouble. We will see a few illustrations of this phenomenon in what follows.

Generalizing some useful concepts from the usual DG theory \cite{KeICM}, we make the following definitions.

\begin{defn}\label{def-finite-cell} Let $A$ be a $p$-DG algebra, and $K$ be a (left or right) $p$-DG module. 
\begin{enumerate}
\item[(i)] The module $K$ is said to satisfy \emph{property (P)} if there exists an increasing, possibly infinite, exhaustive $\dif_K$-stable filtration $F^\bullet$, such that each subquotient $F^{\bullet}/F^{\bullet -1}$ is isomorphic to a direct sum of $p$-DG direct summands of $A$. 
\item[(ii)]The module $K$ is called a \emph{finite cell module}, if it satisfies property (P), and as an $A$-module, it is finitely generated (necessarily projective by the property (P) requirement).
\end{enumerate}
\end{defn}

Property-(P) modules are the analogues of projective modules in usual homological algebra. For instance, the morphism spaces from a property-(P) module to any $p$-DG module coincide in both the homotopy and derived categories.

It is a theorem \cite[Theorem 6.6]{QYHopf} that there are always sufficiently many property-(P) modules: for any $p$-DG module $M$, there is a surjective quasi-isomorphism 
\begin{equation}\label{eqn-bar-resolution}
\mathbf{p}(M)\lra M
\end{equation} 
of $p$-DG modules, with $\mathbf{p}(M)$ satisfying property (P).  We will usually refer to such a property-(P) replacement $\mathbf{p}(M)$ for $M$ as a \emph{bar resolution}. The proof of its existence is similar to that of the usual (simplicial) bar resolution for DG modules over DG algebras. We omit the proof here, but an example of such a construction will be used in the proof of Theorem \ref{thm-pdgjw}. 

In a similar vein, finite cell modules play the role of finite projective modules in usual homological algebra. We give an example of a finite cell module whose construction mimics the notion of ``one-sided twisted complexes'' for DG modules.

\begin{example}\label{eg-finite-cell}
 If $A$ is a $p$-DG algebra, then $A$ itself is finite cell. Now, if $a_0\in A$ is an element such that $\dif^{p-1}(a_0)=0$, then one can easily check that the differential rule on the rank-two free module $Ax\oplus Ay$
\[
 \dif(ax):=\dif(a)x+aa_0y, \quad \dif (ay):=\dif(a)y
\]
defines the structure of a left $p$-DG module which is finite cell. We will schematically depict this module as
\[
Ax \stackrel{a_0}{\lra} Ay,
\]
even though $a_0$ does not necessarily commute with differentials. One should understand this diagram as a filtered $p$-DG module, with $F^0=Ay$ and $F^1=Ax\oplus Ay$.
\end{example}

A $p$-DG bimodule $_AM_B$ over two $p$-DG algebras $A$ and $B$ is a $p$-DG module over $A\otimes B^{\mathrm{op}}$. One has the associated derived tensor product functor
\begin{equation}\label{eqn-def-derived-tensor}
M\otimes_B^{\mathbf{L}}(-): \mathcal{D}(B)\lra \mathcal{D}(A), \quad N\mapsto \mathbf{p}(M)\otimes_B N,
\end{equation}
where $\mathbf{p}(M)$ is a bar resolution for $M$ as a $p$-DG module over $A\otimes B^{\mathrm{op}}$

One useful fact about such functors is the following theorem, whose proof can be found in \cite[Section 8]{QYHopf}.

\begin{thm}\label{thm-qis-functors}
Let $f:{M_1}\lra M_2$ be a quasi-isomorphism of $p$-DG bimodules. Then $f$ descends to an isomorphism of the induced derived tensor product functors. \hfill$\square$
\end{thm}

As the last example of adapting the usual DG theory to the present situation, one may define for the triangulated category $\mathcal{D}(A)$ its \emph{Grothendieck group} $K_0(A)$ as usual. What we gain as dividend in the current situation is that, since $\mathcal{D}(A)$ is a ``categorical module'' over $\Bbbk[\dif]/(\dif^p)\!-\!\underline{\mathrm{gmod}}$, the abelian group $K_0(A)$ naturally has a module structure over the auxiliary cyclotomic ring  at a $p$th root of unity, which was defined in equation \eqref{eqn-aux-ring} of the previous Section:
\begin{equation}\label{eqn-aux-ring-2}
\mathbb{O}_p:=K_0(\Bbbk[\dif]/(\dif^p)\!-\!\underline{\mathrm{gmod}}).
\end{equation}
This explains the diagram \eqref{eqn-categorical-Op-mod} in the Introduction.

The Grothendieck groups are the algebraic invariants that we will be interested in for certain $p$-DG algebras in this note. As in the DG case, in order to avoid trivial cancellations one should restrict the class of objects used to define $K_0(A)$. It turns out that a good condition is the \emph{compactness} of a $p$-DG module. However, we will not recall this definition, but refer the reader to \cite[Section 7]{QYHopf} for the details.

\subsection{The $p$-DG algebra $A_n^!$.} As an illustration of the general theory discussed above, we now present a $p$-DG algebra that is relevant to quantum topology and easy to define.

Let $\Bbbk$ be a field of any characteristic $p\geq 0$ (later we will only focus on the case $p>0$). Let $n$ be a natural number greater than or equal to two, and $Q_n$ be the following quiver:
\begin{equation}\label{quiver-Q}
\xymatrix{
 \overset{1}{\circ} &
 \cdots \ltwocell{'}&
 \overset{i-1}{~\circ~}\ltwocell{'}&
 \overset{i}{\circ}\ltwocell{'}&
 \overset{i+1}{~\circ~}\ltwocell{'}&
 \cdots \ltwocell{'}&
 \overset{n}{\circ}\ltwocell{'}
 }
\end{equation}
Let $\Bbbk Q_n$ be the path algebra associated to $Q_n$ over the ground field. Following Khovanov-Seidel \cite{KS}, we use, for instance, the symbol $(i|j|k)$, where $i,j,k$ are vertices of the quiver $Q_n$, to denote the path which starts at a vertex $i$, then goes through $j$ (necessarily $j=i\pm 1$) and ends at $k$.
The composition of paths is given by
\begin{equation}
(i_i|i_2|\cdots|i_r)\cdot (j_1|j_2|\cdots|j_s)=
\left\{
\begin{array}{ll}
(i_i|i_2|\cdots|i_r|j_2|\cdots|j_s) & \textrm{if $i_r=j_1$,}\\
0 & \textrm{otherwise,}
\end{array}
\right.
\end{equation}
where $i_1,\dots, i_r$ and $j_1, \dots, j_s$ are sequences of neighboring vertices in $Q_n$.

\begin{defn}\label{def-algebra-An}
The algebra $A_n^!$ is the quotient of the path algebra $\Bbbk Q_n$ by the relations
\[
(1|2|1)=0 ~~ \text{ and } ~~ (i|i-1|i)=(i|i+1|i) ~~ \text{ for } ~~ i=2,\dots, n-1.
\]
\end{defn}
It should be mentioned that $A_n^!$ is isomorphic to the endomorphism algebra of a minimal projective generator of the category $\mathcal{O}_{1,n-1}(\mathfrak{gl}_n)$ considered in \cite{BFK}. 
The sum of categories $\oplus_{i=0}^n \mathcal{O}_{i,n-i}(\mathfrak{gl}_n)$ has been used to categorify $V_1^{\otimes n}$ in \cite{BFK}. Moreover, the algebra $A_n^!$ is Koszul, whose quadratic dual is isomorphic to the algebra $A_n$ considered in \cite{KS}.  

This algebra arises naturally from the truncated polynomial algebra $C_n:=\Bbbk[x]/(x^n)$.
For $i=1,\ldots, n$, there are ideals\footnote{The ideals $x^{n-i}C_n$ ($i=1,\dots, n$) are isomorphic to the quotient modules $C/(x^i)$, only as $C$-modules but not as $p$-DG modules. However, the results obtained below will not be affected by using the quotient modules instead: the differentials obtained in different interpretations are conjugate to each other by the involution on $A_n^!$ which reverses arrow directions.}
\begin{equation*}
G^i=x^{n-i}C_n \{ i-n \},
\end{equation*}
where the curly bracket $\{i\}$ stands for a graded module with the grading shifted up by $i$. A routine calculation shows that
\begin{equation*}
A_n^! \cong \END_{C_n}(\bigoplus_{i=1}^n G^i),
\end{equation*}
where the capital $\END_{C_n}$ stands for the space of all endomorphisms of a finite-dimensional $C_n$-module, not necessarily grading preserving.  
The path $(i|i+1)$ gets sent to the homomorphism $G^{i+1} \rightarrow G^i$ which maps $x^{n-i-1}$ to $x^{n-i}$.
The path $(i|i-1)$ gets sent to the homomorphism $G^i \rightarrow G^{i+1}$ which maps $x^{n-i}$ to $x^{n-i}$. The endomorphism algebra inherits a differential determined, for any $\phi \in \END_{C_n}(\bigoplus_{i=1}^n G^i)$ and $z\in \bigoplus_{i=1}^n G^i$, by
\[
\dif(\phi)(z)=\dif(\phi(z))-\phi(\dif(z)).
\]

This realization of $A_n^!$ endows the algebra with a $p$-DG structure.
On $C_n$ there is a differential $ \partial $ given by $\partial(x)=x^2$.  It is clear that each $G^i$ is a $p$-DG submodule of $C_n$.
This gives $A_n^!$ the structure of a $p$-DG algebra where the differential is given by
\[
\partial(i|i+1)=(i|i+1|i|i+1), \quad \quad \partial(i|i-1)=0.
\]

Since the algebra $A_n^!$ is a finite-dimensional Koszul algebra with $p$-differential acting trivially on the degree zero component, the computation of the Grothendieck group of the derived category of $A_n^!$ reduces to the computation of the Grothendieck without the presence of the differential calculated in Proposition ~\ref{K0generic}.  See ~\cite[Corollary 2.18]{EQ1} for more details.

\begin{prop} There is an isomorphism of $\mathbb{O}_p$-modules
$$K_0(\mc{D}(A_n^!)) \cong V_1^{\otimes n} [-n+2].$$
\end{prop}

For each $i=1,\ldots,n$ there is a simple left $A_n^!$-module $L_i$ spanned by $1 \in \Bbbk$ in degree zero such that $(j)1=\delta_{i,j} 1$.  It is a $p$-DG module with trivial differential.  Similarly one may define the simple right $A_n^!$-module ${}_i L$.  In the Grothendieck group the image of the module $L_i$ gets identified with $l_i \in V_1^{\otimes n} [-n+2]$ defined earlier.

\subsection{A braid group action.}

For $i=1, \ldots, n-1$ define derived functors $\mf{U}_i$ as follows:
\begin{equation*}
\mf{U}_i \colon \mc{D}(A_n^!) \rightarrow \mc{D}(A_n^!) \hspace{.5in}
\mf{U}_i(M) = (L_i \otimes {}_i L \otimes_{A_n^!}^\mathbf{L} M)[-1]\{-1\}.
\end{equation*}

The following result shows that these functors define a (weak) categorical Temperley-Lieb algebra action on the derived category $\mathcal{D}(A_n^!)$.

\begin{thm}
[\cite{QiSussan}]
\label{thm-TL-action}
For $ i=1, \ldots, n-1 $, the functors $ \mf{U}_i $ are self-biadjoint, and they enjoy the following functor-isomorphism relations.
\begin{enumerate}
\item[(i)] $ \mf{U}_i \circ \mf{U}_i \cong \mf{U}_i [-1]\lbrace -1 \rbrace \oplus \mf{U}_i[1] \lbrace 1 \rbrace $,
\item[(ii)] $ \mf{U}_i \circ \mf{U}_j \cong \mf{U}_j \circ \mf{U}_i $ for $ |i-j|>1 $,
\item[(iii)] $ \mf{U}_i \circ \mf{U}_j \circ \mf{U}_i \cong \mf{U}_i $ for $ |i-j|=1 $.
\end{enumerate}
\end{thm}

These categorical equations can be roughly explained by the fact that the $p$-DG modules $L_i$ are \emph{spherical objects} in the $p$-DG derived category $\mathcal{D}(A_n^!)$. Spherical objects in the usual DG case were studied in \cite{AnnoL,SeidelThomas,KS}.

To convert spherical objects into braid group generators, we mimic the construction in \cite{KS} in the $p$-DG context. Consider the following canonical maps:
\begin{equation}
\epsilon_1 \colon L_i\o {_iL}[-2]\{-2\}\cong L_i\o \RHOM_{A_n^!}(L_i,A_n^!) \stackrel{\lambda_1}{\lra} A_n^!, \label{eqn-biadjunction-1}
\end{equation}
\begin{equation}
\eta_2 \colon A_n^! \stackrel{\lambda_2}{\lra} \RHOM_\Bbbk (\RHOM_{A_n^!}(A_n^!,{L_i}),{L_i}) \cong L_i\o {_iL}. \label{eqn-biadjunction-4}
\end{equation}
We use them to define \emph{filtered $p$-DG bimodules} and derived functors.

\begin{defn}\label{def-twist-functors}For each $i\in \{1,\dots, n-1\}$, we define the functors $\mf{T}_i$ and $\mf{T}_i^\prime$ as follows.
\begin{enumerate}
\item[(i)] The functor $\mf{T}_i: \mc{D}(A_n^!)\lra \mc{D}(A_n^!)$ is given by the derived tensor product with the cocone of the bimodule adjunction map \eqref{eqn-biadjunction-4}
\[
 A_n^! \xrightarrow{\lambda_2[-1]} \RHOM_\Bbbk (\RHOM_{A_n^!}(A_n^!,{L_i}),{L_i}) [-1] \cong \mf{U}_i \{1\},
\]

\item[(ii)]The functor $\mf{T}_i^\prime: \mc{D}(A_n^!)\lra \mc{D}(A_n^!)$ is given by the derived tensor product with the cone of the bimodule adjunction map \eqref{eqn-biadjunction-1}:
\[
\mf{U}_i\{-1\} \cong L_i \o \RHOM_{A_n^!}(L_i,A_n^!)[1] \xrightarrow{\lambda_1[-1]} A_n^!,
\]
\end{enumerate}
\end{defn}

\begin{thm}
\cite{QiSussan}
\label{thm-braidrelations}
The functors $ \mf{T}_i, \mf{T}_i' $ for $ i=1, \ldots, n-1 $ satisfy braid relations, i.e., there are isomorphisms of functors
\begin{enumerate}
\item $ \mf{T}_i \circ \mf{T}_i' \cong \Id \cong \mf{T}_i' \circ \mf{T}_i $
\item $ \mf{T}_i \circ \mf{T}_j  \cong \mf{T}_j \circ \mf{T}_i $ if $ |i-j|>1 $
\item $ \mf{T}_i \circ \mf{T}_{i+1} \circ \mf{T}_i \cong \mf{T}_{i+1} \circ \mf{T}_{i} \circ \mf{T}_{i+1} $.
\end{enumerate}
\end{thm}

The proofs of both Theorems  \ref{thm-TL-action} and \ref{thm-braidrelations} are applications of Theorem \ref{thm-qis-functors} plus explicit comparisons of the bimodules involved.

\section{A categorification of quantum $\mathfrak{sl}_2$ at prime roots of unity}\label{sec-cat-sl(2)}
\subsection{The $p$-DG nilHecke category.}The ring $C_n$ considered in the previous section has a natural topological origin as $C_n \cong \mH^*(\mathbb{P}^{n-1},\Bbbk)$. Next, we will consider other similarly defined rings from topology, and equip them with $p$-differentials over a field of characteristic $p$.

Let $Fl_n$ be the flag variety of complete flags in $\mathbb{C}^n$. It is well known that the $GL_n(\mathbb{C})$-equivariant cohomology of $Fl_n$ is identified with
\begin{equation}\label{eqn-equ-coho-FL}
\mH_{GL_n(\mathbb{C})}^*(Fl_n,\Bbbk)\cong \mH^*_{B_n}(\mathrm{pt},\Bbbk)\cong \pol_n,
\end{equation}
where $B_n$ stands for a Borel subgroup of $GL_n(\mathbb{C})$, and $\pol_n$ denotes the algebra of polynomial functions on $n$-letters. As such, $\pol_n$ is naturally a module over
\begin{equation}\label{eqn-equ-coho-pt}
\mH^*_{GL_n(\mathbb{C})}(\mathrm{pt})\cong \sym_n.
\end{equation}
The algebra of symmetric functions $\sym_n$ is nothing but the polynomial algebra on the usual elementary symmetric functions $e_i$ ($i=1,\dots, n$) on $n$-letters. We define a $p$-DG algebra structure on $\sym_n$ by setting
\begin{equation}\label{eqn-p-dg-sym}
\dif(e_i):=e_1e_i-(i+1)e_{i+1},
\end{equation}
with it understood that $e_{n+1}=0$. Write the resulting $p$-DG algebra as $(\sym_n,\dif)$.

The \emph{nilHecke algebra on $n$-letters} $\nh_n$, by definition, is the endomorphism algebra of the module $\pol_n$ over $\sym_n$:
\begin{equation}
\nh_n\cong \END_{\sym_n}(\pol_n).
\end{equation}
It is useful to give a generators-and-relations presentation of $\nh_n$ as follows. Let $x_i$ $(i=1,\dots, n)$ be the operator of multiplication on $\pol_n$ by the $i$th variable, and define $\tau_i$ $(i=1,\dots, n-1)$ to be the $i$th \emph{divided difference operator} given by
\[
\tau_i(f)=\dfrac{f-{^{i}f}}{x_i-x_{i+1}}.
\]
Here $f$ is any polynomial in $\pol_n$, and $^{i}f$ stands for $f$ with the $i$th and $i+1$st variables exchanged.  A full set of relations satisfied by the generators is given by
\begin{equation}\label{eqn-NH-relation}
\begin{gathered}
x_ix_j=x_jx_i, \quad \quad x_i\tau_j=\tau_jx_i~(|i-j|>1),\quad \quad \tau_i\tau_j=\tau_j\tau_i~(|i-j|>1),\\
\tau_i^2=0,
\quad x_i\tau_i-\tau_ix_{i+1}=1=\tau_ix_i-x_{i+1}\tau_i,
\quad \tau_{i}\tau_{i+1}\tau_1=\tau_{i+1}\tau_i\tau_{i+1}.
\end{gathered}
\end{equation}
Define
\begin{equation*}
\nh:=\bigoplus_{n\in \N}\nh_n\gmod.
\end{equation*}
A key property of $\nh$ is that it is a \emph{monoidal category}, with the monoidal structure coming from a natural inclusion map
\begin{equation}\label{eqn-nh-inclusions}
\nh_n\otimes \nh_m\subset \nh_{n+m}.
\end{equation}
It is quite surprising to us that, not until the recent work of Khovanov, Lauda and Rouquier \cite{KL1, Lau1, Rou2}, the following simple relationship between quantum $\mathfrak{sl}_2$ and nilHecke algebras, the latter dating back to the work of Bernstein-Gelfand on Schubert calculus, was not stated anywhere in the literature.

\begin{thm}[\cite{KL1, Lau1, Rou2}]There is an isomorphism of $\Z[q,q^{-1}]$-bialgebras
\[
U_q^-(\mathfrak{sl}_2)\cong K_0(\nh),
\]
with the symbols of induction and restriction functors on the right hand side corresponding to the multiplication and comultiplication on $U_q^-(\mathfrak{sl}_2)$.
\end{thm}

To categorify half of the small quantum $\mathfrak{sl}_2$, we combine this construction with hopfological algebra of $p$-differentials. To do so, define $\pol_n$ to be the $p$-DG algebra with the $p$-differential $\dif(x_i)=x_i^2$ ($i=1,\dots, n$). Following \cite[Section 3]{KQ}, define a $p$-DG module structure on the rank-one $\pol_n$ module $\mathcal{P}_n$ as follows:
\begin{equation}\label{eqn-pdg-poln}
\mathcal{P}_n:=\pol_n\cdot v_0,\quad \quad \dif(v_0)=-\sum_{i=1}^{n}(n-i)x_iv_0.
\end{equation}
Upon restriction, we regard $\mathcal{P}_n$ as a $p$-DG module over $(\sym_n,\dif)$. It is shown in \cite[Section 3]{KQ} that, up to an isomorphism, this is the unique $p$-DG structure on $\pol_n$ over $(\sym_n,\dif)$ that makes the module finite cell in all characteristic $p>0$. Furthermore, it also induces a monoidal-structure preserving (see equation \eqref{eqn-nh-inclusions}) differential on the $\nh$ category. A simple computation identifying $\nh_n\cong \END_{\sym_n}(\mathcal{P}_n)$ as $p$-DG algebras yields the following definition.

\begin{defn} \label{def-p-DG-NH}
The $p$-DG monoidal category is the direct sum of $p$-DG algebras
$$(\nh, \dif):=\oplus_{n\in \N}(\nh_n,\dif),$$
where each $\nh_n$ is equipped with the $p$-differential defined on the generators (see equation \eqref{eqn-NH-relation}) by
\[
\dif(x_i)=x_i^2, \quad \quad \dif(\tau_i)=-x_i\tau_i-\tau_ix_{i+1}.
\]
\end{defn}

\begin{thm}[\cite{KQ}]\label{thm-p-DG-NH}There is an isomorphism of $\mathbb{O}_p$-bialgebras
\[
u^-_{\mathbb{O}_p}\cong K_0(\mathcal{D}(\nh)),
\]
under which $F^{(n)}$ is identified with the symbol $[\mathcal{P}_n]$.
\end{thm}

To see that this is indeed the case, we justify the relation $F^{n}=0$ if $n\geq p$. One can readily see that, as a $p$-DG module over $\sym_n$, $\mathcal{P}_n$ has a $\dif$-stable basis consisting of elements in the set
\[
B_n:=\{x_1^{a_1}\dots x_n^{a_n}|0\leq  a_i\leq n-i\}.
\]
The $\Bbbk$-span of $B_n$ forms an $n!$-dimensional $p$-complex, which is acyclic whenever $n\geq p$. Therefore, since $\nh_n\cong \END_{\sym_n}(\mathcal{P}_n)$, $\nh_n$ is an acyclic $p$-DG algebra whenever $n\geq p$. It follows that $\mathcal{D}(\nh_n)\cong 0$ if $n\geq p$, and this categorical vanishing result lifts the above relation $F^{n}=0$ ($n\geq p$) for the small quantum group.

\subsection{The $p$-DG $2$-category $\mathcal{U}$.}
We seek to ``Drinfeld-double'' the one-half construction of the categorified small quantum group following Khovanov-Lauda and Rouquier \cite{Lau1,KL3,Rou2}. To do so for a generic $q$, one takes two copies of the $\nh$-category, and ``glue'' them together with biadjuntions. More specifically, let us recall the diagrammatic definition of $\mathcal{U}$ introduced by Lauda in \cite{Lau1}.

\begin{defn}\label{def-u-dot} The $2$-category $\mathcal{U}$ is an additive graded $\Bbbk$-linear category whose objects $\lambda$ are elements of the weight lattice of $\mathfrak{sl}_2$. The $1$-morphisms are (direct sums of grading shifts of) composites of the generating $1$-morphisms $\1_{\lambda+2} \mathcal{E} \1_\lambda$ and $\1_\lambda \mathcal{F} \1_{\lambda+2}$, for each $\lambda \in \Z$. Each $\1_{\lambda+2} \mathcal{E} \1_\lambda$ will be drawn the same, regardless of the object $\lambda$.

\begin{align*}
\begin{tabular}{|c|c|c|}
	\hline
	$1$-\textrm{morphism Generator} &
	\begin{DGCpicture}
	\DGCstrand(0,0)(0,1)
	\DGCdot*>{0.5}
	\DGCcoupon*(0.1,0.25)(1,0.75){$^\lambda$}
	\DGCcoupon*(-1,0.25)(-0.1,0.75){$^{\lambda+2}$}
    \DGCcoupon*(-0.25,1)(0.25,1.15){}
    \DGCcoupon*(-0.25,-0.15)(0.25,0){}
	\end{DGCpicture}&
	\begin{DGCpicture}
	\DGCstrand(0,0)(0,1)
	\DGCdot*<{0.5}
	\DGCcoupon*(0.1,0.25)(1,0.75){$^{\lambda+2}$}
	\DGCcoupon*(-1,0.25)(-0.1,0.75){$^\lambda$}
    \DGCcoupon*(-0.25,1)(0.25,1.15){}
    \DGCcoupon*(-0.25,-0.15)(0.25,0){}
	\end{DGCpicture} \\ \hline
	\textrm{Name} & $\1_{\lambda+2}\mathcal{E}\1_\lambda$ & $\1_\lambda\mathcal{F}\1_{\lambda+2}$ \\
	\hline
\end{tabular}
\end{align*}

The weight of any region in a diagram is determined by the weight of any single region. When no region is labeled, the ambient weight is irrelevant.

The $2$-morphisms will be generated by the following pictures.
	
\begin{align*}
\begin{tabular}{|c|c|c|c|c|}
  \hline
  \textrm{Generator} &
  \begin{DGCpicture}
  \DGCstrand(0,0)(0,1)
  \DGCdot*>{0.75}
  \DGCdot{0.45}
  \DGCcoupon*(0.1,0.25)(1,0.75){$^\lambda$}
  \DGCcoupon*(-1,0.25)(-0.1,0.75){$^{\lambda+2}$}
  \DGCcoupon*(-0.25,1)(0.25,1.15){}
  \DGCcoupon*(-0.25,-0.15)(0.25,0){}
  \end{DGCpicture}&
  \begin{DGCpicture}
  \DGCstrand(0,0)(0,1)
  \DGCdot*<{0.25}
  \DGCdot{0.65}
  \DGCcoupon*(0.1,0.25)(1,0.75){$^{\lambda}$}
  \DGCcoupon*(-1,0.25)(-0.1,0.75){$^{\lambda-2}$}
  \DGCcoupon*(-0.25,1)(0.25,1.15){}
  \DGCcoupon*(-0.25,-0.15)(0.25,0){}
  \end{DGCpicture} &
  \begin{DGCpicture}
  \DGCstrand(0,0)(1,1)
  \DGCdot*>{0.75}
  \DGCstrand(1,0)(0,1)
  \DGCdot*>{0.75}
  \DGCcoupon*(1.1,0.25)(2,0.75){$^\lambda$}
  \DGCcoupon*(-1,0.25)(-0.1,0.75){$^{\lambda+4}$}
  \DGCcoupon*(-0.25,1)(0.25,1.15){}
  \DGCcoupon*(-0.25,-0.15)(0.25,0){}
  \end{DGCpicture} &
  \begin{DGCpicture}
  \DGCstrand(0,0)(1,1)
  \DGCdot*<{0.25}
  \DGCstrand(1,0)(0,1)
  \DGCdot*<{0.25}
  \DGCcoupon*(1.1,0.25)(2,0.75){$^\lambda$}
  \DGCcoupon*(-1,0.25)(-0.1,0.75){$^{\lambda-4}$}
  \DGCcoupon*(-0.25,1)(0.25,1.15){}
  \DGCcoupon*(-0.25,-0.15)(0.25,0){}
  \end{DGCpicture} \\ \hline
  \textrm{Degree}  & 2   & 2 & -2 & -2 \\
  \hline
\end{tabular}
\end{align*}

\begin{align*}
\begin{tabular}{|c|c|c|c|c|}
  \hline
  \textrm{Generator} &
  \begin{DGCpicture}
  \DGCstrand/d/(0,0)(1,0)
  \DGCdot*>{-0.25,1}
  \DGCcoupon*(1,-0.5)(1.5,0){$^\lambda$}
  \DGCcoupon*(-0.25,0)(1.25,0.15){}
  \DGCcoupon*(-0.25,-0.65)(1.25,-0.5){}
  \end{DGCpicture} &
  \begin{DGCpicture}
  \DGCstrand/d/(0,0)(1,0)
  \DGCdot*<{-0.25,2}
  \DGCcoupon*(1,-0.5)(1.5,0){$^\lambda$}
  \DGCcoupon*(-0.25,0)(1.25,0.15){}
  \DGCcoupon*(-0.25,-0.65)(1.25,-0.5){}
  \end{DGCpicture}&
  \begin{DGCpicture}
  \DGCstrand(0,0)(1,0)/d/
  \DGCdot*<{0.25,1}
  \DGCcoupon*(1,0)(1.5,0.5){$^\lambda$}
  \DGCcoupon*(-0.25,0.5)(1.25,0.65){}
  \DGCcoupon*(-0.25,-0.15)(1.25,0){}
  \end{DGCpicture}&
  \begin{DGCpicture}
  \DGCstrand(0,0)(1,0)/d/
  \DGCdot*>{0.25,2}
  \DGCcoupon*(1,0)(1.5,0.5){$^\lambda$}
  \DGCcoupon*(-0.25,0.5)(1.25,0.65){}
  \DGCcoupon*(-0.25,-0.15)(1.25,0){}
  \end{DGCpicture}  \\ \hline
  \textrm{Degree} & $1+\lambda$ & $1-\lambda$ & $1+\lambda$ & $1-\lambda$ \\
  \hline
\end{tabular}
\end{align*}
\end{defn}

We will give a list of relations shortly, after we discuss some notation. For a product of $m$ dots on a single strand, we draw a single dot labeled by $m$. Here is the case when $m=4$.

\begin{align*}
\begin{DGCpicture}[scale=.75]
\DGCstrand(0,0)(0,2)
\DGCdot{.3}
\DGCdot{.65}
\DGCdot{1}
\DGCdot{1.35}
\DGCdot*>{1.85}
\end{DGCpicture}
~=~
\begin{DGCpicture}[scale=.75]
\DGCstrand(0,0)(0,2)
\DGCdot{1}[r]{$^4$}
\DGCdot*>{1.85}
\end{DGCpicture}
\end{align*}

A \emph{closed diagram} is a diagram without boundary, constructed from the generators above. The simplest non-trivial closed diagram is a \emph{bubble}, which is a closed diagram without any other closed diagrams inside. Bubbles can be oriented clockwise or counter-clockwise.
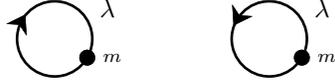
\begin{align*}
\begin{DGCpicture}
\DGCbubble(0,0){0.5}
\DGCdot*<{0.25,L}
\DGCdot{-0.25,R}[r]{$_m$}
\DGCdot*.{0.25,R}[r]{$\lambda$}
\end{DGCpicture}
\qquad \qquad
\begin{DGCpicture}
\DGCbubble(0,0){0.5}
\DGCdot*>{0.25,L}
\DGCdot{-0.25,R}[r]{$_m$}
\DGCdot*.{0.25,R}[r]{$\lambda$}
\end{DGCpicture}
\end{align*}

A simple calculation shows that the degree of a bubble with $m$ dots in a region labeled $\lambda$ is $2(m+1-\lambda)$ if the bubble is clockwise, and $2(m+1+\lambda)$ if the bubble is counter-clockwise. Instead of keeping track of the number $m$ of dots a bubble has, it will be much more illustrative to keep track of the degree of the bubble, which is in $2\Z$. We will use the following shorthand to refer to a bubble of degree $2k$.

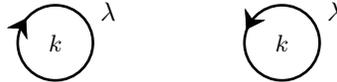
\begin{align*}
\bigcwbubble{$k$}{$\lambda$}
\qquad \qquad
\bigccwbubble{$k$}{$\lambda$}
\end{align*}
This notation emphasizes the fact that bubbles have a life of their own, independent of their presentation in terms of caps, cups, and dots.

Note that $\lambda$ can be any integer, but $m\geq 0$ because it counts dots. Therefore, we can only construct a clockwise (resp. counter-clockwise) bubble of degree $k$ when $k \geq 1-\lambda$ (resp. $k \geq 1+\lambda$). These are called \emph{real bubbles}. Following Lauda, we also allow bubbles drawn as above with arbitrary $k \in \Z$. Bubbles with $k$ outside of the appropriate range are not yet defined in terms of the generating maps; we call these \emph{fake bubbles}. One can express any fake bubble in terms of real bubbles (see Remark \ref{rmk-inf-Grass-relation}).

Now we list the relations. Whenever the region label is omitted, the relation applies to all ambient weights.

\begin{itemize}
\item[(i)] {\bf Biadjointness relations.}
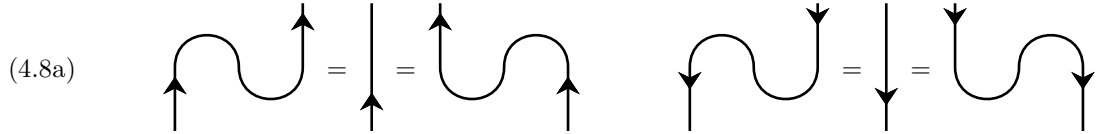
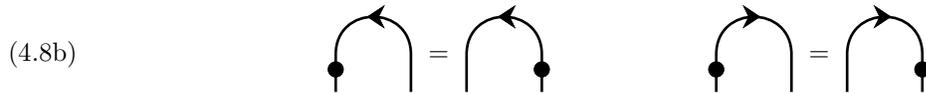
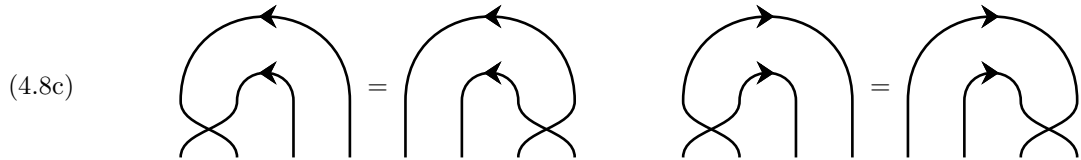
\begin{subequations} \label{biadjoint}
\begin{align} \label{biadjoint1}
\begin{DGCpicture}[scale=0.85]
\DGCstrand(0,0)(0,1)(1,1)(2,1)(2,2)
\DGCdot*>{0.75,1}
\DGCdot*>{1.75,1}
\end{DGCpicture}
~=~
\begin{DGCpicture}[scale=0.85]
\DGCstrand(0,0)(0,2)
\DGCdot*>{0.5}
\end{DGCpicture}
~=~
\begin{DGCpicture}[scale=0.85]
\DGCstrand(2,0)(2,1)(1,1)(0,1)(0,2)
\DGCdot*>{0.75}
%\DGCdot{1.05}[r]{$^m$}
\DGCdot*>{1.75}
%\DGCcoupon*(2.1,0.25)(2.5,0.75){$n$}
\end{DGCpicture}
\qquad \qquad
\begin{DGCpicture}[scale=0.85]
\DGCstrand(0,0)(0,1)(1,1)(2,1)(2,2)
\DGCdot*<{0.75,1}
\DGCdot*<{1.75,1}
%\DGCdot{1.05,1}[r]{$^m$}
%\DGCcoupon*(2.1,0.25)(2.5,0.75){$n$}
\end{DGCpicture}
~=~
\begin{DGCpicture}[scale=0.85]
\DGCstrand(0,0)(0,2)
%\DGCdot{1}[r]{$^m$}
\DGCdot*<{0.5}
%\DGCcoupon*(0.1,0.25)(0.5,0.75){$n$}
\end{DGCpicture}
~=~
\begin{DGCpicture}[scale=0.85]
\DGCstrand(2,0)(2,1)(1,1)(0,1)(0,2)
\DGCdot*<{0.75}
%\DGCdot{1.05}[r]{$^m$}
\DGCdot*<{1.75}
%\DGCcoupon*(2.1,0.25)(2.5,0.75){$n$}
\end{DGCpicture}
\end{align}

\begin{align} \label{biadjointdot}
\begin{DGCpicture}
\DGCstrand(0,0)(0,.5)(1,.5)/d/(1,0)/d/
\DGCdot*<{1}
\DGCdot{.3,1}
\end{DGCpicture}
~=~
\begin{DGCpicture}
\DGCstrand(0,0)(0,.5)(1,.5)/d/(1,0)/d/
\DGCdot*<{1}
\DGCdot{.3,2}
\end{DGCpicture}
\qquad \qquad \qquad
\begin{DGCpicture}
\DGCstrand(0,0)(0,.5)(1,.5)/d/(1,0)/d/
\DGCdot*>{1}
\DGCdot{.3,1}
\end{DGCpicture}
~=~
\begin{DGCpicture}
\DGCstrand(0,0)(0,.5)(1,.5)/d/(1,0)/d/
\DGCdot*>{1}
\DGCdot{.3,2}
\end{DGCpicture}
\end{align}

\begin{align} \label{biadjointcrossing}
\begin{DGCpicture}[scale=0.75]
\DGCstrand(0,0)(1,1)/u/(2,1)/d/(2,0)/d/
\DGCdot*<{1.5}
\DGCstrand(1,0)(0,1)/u/(3,1)/d/(3,0)/d/
\DGCdot*<{2.5}
\end{DGCpicture}
~=~
\begin{DGCpicture}[scale=0.75]
\DGCstrand(0,0)(0,1)(3,1)/d/(2,0)/d/
\DGCdot*<{2.5}
\DGCstrand(1,0)(1,1)(2,1)/d/(3,0)/d/
\DGCdot*<{1.5}
\end{DGCpicture}
\qquad \qquad
\begin{DGCpicture}[scale=0.75]
\DGCstrand(0,0)(1,1)/u/(2,1)/d/(2,0)/d/
\DGCdot*>{1.5}
\DGCstrand(1,0)(0,1)/u/(3,1)/d/(3,0)/d/
\DGCdot*>{2.5}
\end{DGCpicture}
~=~
\begin{DGCpicture}[scale=0.75]
\DGCstrand(0,0)(0,1)(3,1)/d/(2,0)/d/
\DGCdot*>{2.5}
\DGCstrand(1,0)(1,1)(2,1)/d/(3,0)/d/
\DGCdot*>{1.5}
\end{DGCpicture}
\end{align}
\end{subequations}

\item[(ii)] {\bf Positivity and Normalization of bubbles.} Positivity states that all bubbles (real or fake) of negative degree should be zero.
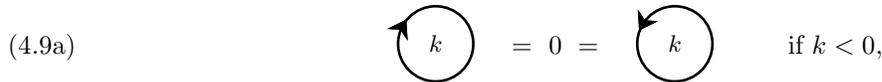
\begin{subequations} \label{negzerobubble}
\begin{align} \label{negbubble}
\bigcwbubble{$k$}{}~=~0~=~\bigccwbubble{$k$}{}
\qquad
\textrm{if}~k<0,
\end{align}

Normalization states that degree $0$ bubbles are equal to the empty diagram (i.e. the identity $2$-morphism of the identity $1$-morphism).

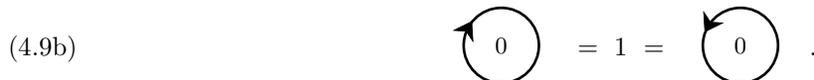
\begin{align} \label{zerobubble}
\bigcwbubble{$0$}{}~=~1~=~\bigccwbubble{$0$}{}~.
\end{align}
\end{subequations}

\item[(iii)] {\bf NilHecke relations.} The upward pointing strands satisfy nilHecke relations as in \eqref{eqn-NH-relation}. Note that, diagrammatically, far-away commuting elements become isotopy relations and are thus built in by default.
\begin{subequations} \label{NHrels}
\begin{align}
\begin{DGCpicture}
\DGCstrand(0,0)(1,1)(0,2)
\DGCdot*>{2}
\DGCstrand(1,0)(0,1)(1,2)
\DGCdot*>{2}
\end{DGCpicture}
=0\ , \quad \quad
\RIII{L}{$0$}{$0$}{$0$}{no} = \RIII{R}{$0$}{$0$}{$0$}{no},
\label{NHrelR3}
\end{align}
\begin{align}
\crossing{$0$}{$0$}{$1$}{$0$}{no} - \crossing{$0$}{$1$}{$0$}{$0$}{no} = \twolines{$0$}{$0$}{no} = \crossing{$1$}{$0$}{$0$}{$0$}{no} -
\crossing{$0$}{$0$}{$0$}{$1$}{no}. \label{NHreldotforce}
\end{align}
\end{subequations}
\item[(iv)] {\bf Reduction to bubbles.} The following equalities hold for all $\lambda \in \Z$.
\begin{subequations} \label{bubblereduction}
\begin{align}
\curl{R}{U}{$\lambda$}{no}{$0$} = -\sum_{a+b=-\lambda} \oneline{$b$}{$\lambda$} \bigcwbubble{$a$}{}\ ,
\end{align}
\begin{align}
\curl{L}{U}{$\lambda$}{no}{$0$} = \sum_{a+b=\lambda} \bigccwbubble{$a$}{$\lambda$} \oneline{$b$}{no} \ .
\end{align}
\end{subequations}
These sums only take values for $a,b \geq 0$. Therefore, when $\lambda \neq 0$, either the right curl or the left curl is zero.
\item[(v)] {\bf Identity decomposition.} The following equations hold for all $\lambda \in \Z$.
\begin{subequations} \label{IdentityDecomp}
\begin{align}
\begin{DGCpicture}
\DGCstrand(0,0)(0,2)
\DGCdot*>{1}
\DGCdot*.{1.25}[l]{$\lambda$}
\DGCstrand(1,0)(1,2)
\DGCdot*<{1}
\DGCdot*.{1.25}[r]{$\lambda$}
\end{DGCpicture}
~=~-~
\begin{DGCpicture}
\DGCstrand(0,0)(1,1)(0,2)
\DGCdot*>{0.25}
\DGCdot*>{1}
\DGCdot*>{1.75}
\DGCstrand(1,0)(0,1)(1,2)
\DGCdot*.{1.25}[l]{$\lambda$}
\DGCdot*<{0.25}
\DGCdot*<{1}
\DGCdot*<{1.75}
\end{DGCpicture}
~+~
\sum_{a+b+c=\lambda-1}~
\cwcapbubcup{$a$}{$b$}{$c$}{$\lambda$} \ , \label{IdentityDecompPos}
\end{align}
\begin{align}
\begin{DGCpicture}
\DGCstrand(0,0)(0,2)
\DGCdot*<{1}
\DGCdot*.{1.25}[l]{$\lambda$}
\DGCstrand(1,0)(1,2)
\DGCdot*>{1}
\DGCdot*.{1.25}[r]{$\lambda$}
\end{DGCpicture}
~=~-~
\begin{DGCpicture}
\DGCstrand(0,0)(1,1)(0,2)
\DGCdot*<{0.25}
\DGCdot*<{1}
\DGCdot*<{1.75}
\DGCstrand(1,0)(0,1)(1,2)
\DGCdot*.{1.25}[l]{$\lambda$}
\DGCdot*>{0.25}
\DGCdot*>{1}
\DGCdot*>{1.75}
\end{DGCpicture}
~+~
\sum_{a+b+c=-\lambda-1}~
\ccwcapbubcup{$a$}{$b$}{$c$}{$\lambda$} \ . \label{IdentityDecompNeg}
\end{align}
\end{subequations}
The sum in the first equality vanishes for $\lambda \leq 0$, and the sum in the second equality vanishes for $\lambda \geq 0$.

The terms on the right hand side form a collection of orthogonal idempotents. 
\end{itemize}

\begin{rem}[Infinite Grassmannian relations]\label{rmk-inf-Grass-relation}  These relations, which follow from the above defining relations, can be expressed most succinctly in terms of generating functions.

\begin{equation}\label{eqn-infinite-Grassmannian}
\left( \cwbubble{$0$}{}+t~\cwbubble{$1$}{}+t^2~\cwbubble{$2$}{}+\ldots \right) \cdot
\left( \ccwbubble{$0$}{}+t~\ccwbubble{$1$}{}+t^2~\ccwbubble{$2$}{}+\ldots \right)  =  1~.
\end{equation}

The cohomology ring of the ``infinite dimensional Grassmannian" is the ring $\Lambda$ of symmetric functions. Inside this ring, there is an analogous relation $e(t)h(t)=1$, where $e(t) = \sum_{i \ge 0} (-1)^i e_i t^i$ is the total Chern class of the tautological bundle, and $h(t) = \sum_{i \ge 0} h_i t^i$ is the total Chern class of the dual bundle. Lauda has proved that the bubbles in a single region generate an algebra inside $\mathcal{U}$ isomorphic to $\Lambda$.

Looking at the homogeneous component of degree $m$, we have the following equation.
\begin{align}
\sum_{a + b = m} \bigcwbubble{$a$}{} \bigccwbubble{$b$}{} = \delta_{m,0} \label{infgrass}
\end{align}
Because of the positivity of bubbles relation, this equation holds true for any $m \in \Z$, and the sum can be taken over all $a,b \in \Z$.

Using these equations one can express all (positive degree) counter-clockwise bubbles in terms of clockwise bubbles, and vice versa.
Consequentially, all fake bubbles can be expressed in terms of real bubbles.
\end{rem}

\begin{defn}\label{def-special-dif}
Let $\dif$ be the derivation defined on the $2$-morphism generators of $\mathcal{U}$ as follows.
\[
\begin{array}{c}
\dif \left( \onelineshort{$1$}{no} \right) =  \onelineshort{$2$}{no} , \quad \quad \quad
\dif \left( \crossing{$0$}{$0$}{$0$}{$0$}{no} \right) =  \twolines{$0$}{$0$}{no} -2 \crossing{$1$}{$0$}{$0$}{$0$}{no} ,\\ \\
\dif \left( \onelineDshort{$1$}{no} \right) =  \onelineDshort{$2$}{no} , \quad \quad \quad
\dif \left( \crossingD{$0$}{$0$}{$0$}{$0$}{no} \right) = - \twolinesD{$0$}{$0$}{no} -2 \crossingD{$0$}{$1$}{$0$}{$0$}{no},
\end{array}
\]
\[
\begin{array}{c}
\dif \left( \cappy{CW}{$0$}{no}{$\lambda$} \right)  =  \cappy{CW}{$1$}{no}{$\lambda$}- \cappy{CW}{$0$}{CW}{$\lambda$},\quad \quad \quad
\dif \left( \cuppy{CW}{$0$}{no}{$\lambda$} \right) = (1-\lambda) \cuppy{CW}{$1$}{no}{$\lambda$},\\ \\
\dif \left( \cuppy{CCW}{$0$}{no}{$\lambda$} \right) =  \cuppy{CCW}{$1$}{no}{$\lambda$} + \cuppy{CCW}{$0$}{CW}{$\lambda$},\quad \quad \quad
\dif \left( \cappy{CCW}{$0$}{no}{$\lambda$} \right) = (\lambda+1 ) \cappy{CCW}{$1$}{no}{$\lambda$}.
\end{array}
\]
\end{defn}

\begin{thm}[\cite{EQ1}]\label{thm-itsanalghom}
There is an isomorphism of $\mathbb{O}_p$-algebras
$$\dot{u}_{\mathbb{O}_p} \lra K_0(\mathcal{D}(\mathcal{U})),$$
sending $1_{\lambda+2}E1_\lambda $ to $[\1_{\lambda+2}\mathcal{E\1_\lambda}]$ and $1_\lambda F 1_{\lambda+2} $ to $ [\1_\lambda \mathcal{F}\1_{\lambda+2}]$ for any weight $\lambda \in \Z$.
\end{thm}

Without the presence of the differential, Khovanov, Lauda, Mackaay and Sto\v{s}i\'{c} (KLMS) has given a diagrammatic presentation of the idempotent completion $\dot{\mathcal{U}}$ of $\mathcal{U}$. By construction, $\dot{\mathcal{U}}$ is Morita equivalent to ${\mathcal{U}}$, and they both categorify quantum $\mathfrak{sl}_2$ at generic $q$ values.

The KLMS calculus $\dot{\mathcal{U}}$ has been equipped with a compatible differential derived from $\mathcal{U}$. We will refer the reader to \cite{EQ2} for the details. However, unlike the abelian case, the $p$-DG derived categories are drastically different! There is a functor
\[
\mathcal{J}: \mathcal{H}(\mathcal{U}) \lra \mathcal{H}(\dot{\mathcal{U}}),
\]
which is given by tensor product with $\dot{\mc{U}}$ regarded as a $(\dot{\mathcal{U}},{\mathcal{U}})$-bimodule with a compatible differential. This functor, not surprisingly, is an equivalence of homotopy categories. However, under localization, it is no longer an equivalence, but instead is a fully-faithful embedding of derived categories:
\[
\mathcal{J}: \mathcal{D}(\mathcal{U}) \lra \mathcal{D}(\dot{\mathcal{U}}).
\]
This derived embedding categorifies the embedding of $\dot{u}_{\mathbb{O}_p}$ into the BLM form $\dot{U}_{\mathbb{O}_p}$ for quantum $\mathfrak{sl}(2)$, the latter generated by all divided powers $E^{(n)},~F^{(n)}$ ($n\in \N$).

\begin{rem}
When $p=2$, the constructions in this section can be lifted to characteristic zero, using the usual DG theory. The nilHecke category should then be replaced by the odd nilHecke algebra of \cite{EKL}. We refer the reader to \cite{EllisQ} for the details. But a natural question still remains: Can one lift the constructions in positive characteristics to characteristic zero? A fundamental problem one needs to tackle is to have a characteristic zero categorification of the ring of cyclotomic integers, and then build nilHecke-like objects in that category.
\end{rem}

\section{A categorification of the Jones-Wenzl projector}\label{sec-jones-wenzl}

\subsection{Categorification at a generic $q$.}
We present here a special case of a functor categorifying the Jones-Wenzl projector first considered in ~\cite{FSS} in the context of category $\mathcal{O}$.  
A further specialization has been studied in greater detail in ~\cite{StroppelSussan}, where the Koszul dual functor is computed and shown to be related to the independent constructions of Cooper and Krushkal ~\cite{CoKr} and Rozansky ~\cite{RoJW}.
What follows below is the construction for the category categorifying the second lowest weight space.
Details for other weight spaces will appear in \cite{StroppelSussanColorJ}.

First we ignore the differential thus categorifying the Jones-Wenzl projector for generic $q$.
For ease of notation set
\begin{equation*}
C:= \Bbbk[x]/(x^n)
\hspace{.5in}
A:=A_n^!.
\end{equation*}

\begin{prop}
\label{K0generic}
The Grothendieck groups of finitely generated graded modules over $A$ and $C$ categorify weight spaces of $V_1^{\otimes n}$ and $V_n$ respectively.
\begin{enumerate}
\item $K_0(A\gmod) \cong V_1^{\otimes n} [-n+2]$.
\item $K_0(C\gmod) \cong V_n [-n+2]$.
\end{enumerate}
\end{prop}

\begin{proof}
It is easy to see that the projective modules $A(i)$ for $i=1,\ldots,n$ are mutually non-isomorphic and exhaust the list of indecomposable $A$-modules.

The second part is even easier since $C$ is a graded local algebra.
\end{proof}

The connection between these categories is given by the Soergel functor and its adjoint.
It is easy to see that $\END_{A}(A(n)) \cong C$.

Consider the following two functors. Let
\begin{equation*}
\Pi  \colon A\gmod \rightarrow C\gmod
\end{equation*}
be the exact functor given by
\begin{equation*}
\Pi (M) = (n)A \otimes_A M,
\end{equation*}
which automatically extends to the derived categories. And let
\begin{equation*}
\mathrm{I} \colon \mc{D}^<(C\gmod) \rightarrow \mc{D}^<(A\gmod)
\end{equation*}
be the derived functor given by
\begin{equation*}
\mathrm{I}(M) =  A(n) \otimes_{C}^\mathbf{L} M.
\end{equation*}

Then define the categorified Jones-Wenzl projector to be
\begin{equation*}
\mf{P} =\mathrm{I} \circ \Pi \colon \mc{D}^<(A\gmod) \rightarrow \mc{D}^<(A\gmod).
\end{equation*}

In the next result, we construct a projective resolution of $(n)A$ as a graded $(C,A)$-bimodule (recall that unadorned tensor product $\otimes$ is always taken over the ground field).

\begin{prop}\label{projresofbigproj}
The following complex of $(C,A)$-bimodules
\begin{equation*}
\cdots \longrightarrow C \otimes (n)A \{ 2n+2 \} \stackrel{h_3}{\longrightarrow} C \otimes (n)A \{ 2n \}  \stackrel{h_2}{\longrightarrow}  C \otimes (n)A
\{ 2 \} \stackrel{h_1}{\longrightarrow}  C \otimes (n)A
 \stackrel{h_0}{\longrightarrow} (n)A
\end{equation*}
is exact, where the differentials are bimodule maps which are defined on the generators by
\begin{equation*}
h_k(1 \otimes (n))=
\begin{cases}
(n) & \text{ if } k=0 \\
x \otimes (n) - 1 \otimes (n|n-1|n)  & \text{ if } k=1,3,\ldots \\
\sum_{j=0}^{n-1}
x^j \otimes (n|n-1|n)^{n-1-j}  & \text{ if } k=2,4,\ldots.
\end{cases}
\end{equation*}
\end{prop}
\begin{proof} It is clear that the square of the differential is zero. Thus it suffices to show that the complex is exact.

To do so, notice that, as a left module over $C$, $(n)A$ decomposes, up to grading shifts, into a direct sum of ideals
\begin{equation}\label{eqn-C-summand-of-A(n)}
(n)A\cong C\oplus xC\oplus \cdots \oplus x^{n-1}C.
\end{equation}
Replacing $(n)A$ in each term of the complex by this decomposition, and we are reduced to showing that, for each $i\in \{0,1,\dots, n-1\}$, the complex
\begin{equation}\label{eqn-generali-term}
\cdots \longrightarrow C \otimes x^iC \{ 2n+2 \} \stackrel{h_3}{\longrightarrow} C \otimes x^iC \{ 2n \}  \stackrel{h_2}{\longrightarrow}  C \otimes x^iC
\{ 2 \} \stackrel{h_1}{\longrightarrow}  C \otimes x^iC
 \stackrel{h_0}{\longrightarrow} x^iC
\end{equation}
is exact. We will first prove it for $i=0$, in this case, the complex is isomorphic to
\begin{equation}\label{eqn-i-zero-term}
\cdots \longrightarrow \dfrac{\Bbbk[x,y]}{(x^n,y^n)} \{ 2n+2 \} \stackrel{h_3}{\longrightarrow} \dfrac{\Bbbk[x,y]}{(x^n,y^n)} \{ 2n \}  \stackrel{h_2}{\longrightarrow} \dfrac{\Bbbk[x,y]}{(x^n,y^n)}
\{ 2 \} \stackrel{h_1}{\longrightarrow}  \dfrac{\Bbbk[x,y]}{(x^n,y^n)}
 \stackrel{h_0}{\longrightarrow} \dfrac{\Bbbk[z]}{(z^n)},
\end{equation}
and the boundary maps are now identified with $h_0(x)=z$, $h_0(y)=z$, and
\[
h_k(1)=
\begin{cases}
x-y  & \text{ if } k=1,3,\ldots \\
\sum_{j=0}^{n-1}x^jy^{n-1-j} & \text{ if } k=2,4,\ldots.
\end{cases}
\]
The exactness of this complex is clear.

For $i>0$, notice that $x^iC\cong C/(x^{n-i})$. If one chops off the right most term of equation \eqref{eqn-generali-term}, that complex has higher homology groups isomorphic to $\mathrm{Tor}^\bullet_C(C,C/(x^{n-i}))=0$ ($\bullet > 0$), and the zeroth homology isomorphic to $C/(x^{n-i})$. The result follows.
\end{proof}

Now tensor the projective resolution in Proposition~\ref{projresofbigproj} on the left by $A(n)$ over $C$ to get a complex of $(A,A)$-bimodules quasi-isomorphic to the functor $\mf{P}$.

\begin{thm}
\begin{enumerate}
\item
The Jones-Wenzl functor $\mf{P}$ is quasi-isomorphic to the following complex of
$(A, A)$-bimodules
\begin{equation*}
\label{res2}
\cdots \rightarrow A(n) \otimes (n)A \{ 2n+2 \} \stackrel{g_3}{\longrightarrow}  A(n) \otimes (n)A \{ 2n \} \stackrel{g_2}{\longrightarrow}  A(n) \otimes (n)A \{ 2 \} \stackrel{g_1}{\longrightarrow}  A(n) \otimes (n)A
\end{equation*}
where
\begin{equation*}
g_i =
\begin{cases}
(n|n-1|n) \otimes 1 - 1 \otimes (n|n-1|n) & \text{ if i is odd} \\
\sum_{j=0}^{n-1} (n|n-1|n)^j \otimes (n|n-1|n)^{n-1-j} & \text{ if i is even}.
\end{cases}
\end{equation*}
\item The functor $\mf{P}$ is idempotent: $\mf{P} \circ \mf{P} \cong \mf{P}$.
\item The Temperley-Lieb functors $\mf{U}_i$, $i=1,\ldots,n-1$, annhilate the Jones-Wenzl projector:
$$ \mf{U}_i \circ \mf{P} \cong 0 \cong \mf{P} \circ \mf{U}_i .$$
\end{enumerate}
\end{thm}

\begin{proof}
The first statement is clear from the discussion above.

The second statement follows almost directly from the definition of the categorified Jones-Wenzl projector and associativity of (derived) tensor products. The tensor product of the bimodules representing $\mf{P}$ is given by
\begin{align*}
 A(n) \otimes^\mathbf{L}_C (n)A \otimes_A A(n) \otimes^\mathbf{L}_C (n)A  \cong A(n) \otimes^\mathbf{L}_C (n)A(n) \otimes^\mathbf{L}_C (n) A \cong A(n) \otimes^\mathbf{L}_C (n)A .
\end{align*}
The result follows.

The third part is a direct consequence of the first part since
\begin{equation*}
{}_i L \otimes_A^\mathbf{L} A(n) \cong (n)A \otimes_A^\mathbf{L} L_i=0
\end{equation*}
for $i=1,\ldots,n-1$.
\end{proof}

As a $2$-representation of Lauda's category this complex may be rewritten as
\begin{equation*}
\left(\cdots \stackrel{f_4}{\longrightarrow} \mc{EF} \{ 2n+2 \} \stackrel{f_3}{\longrightarrow} \mc{EF} \{ 2n \} \stackrel{f_2}{\longrightarrow} \mc{EF} \{ 2 \} \stackrel{f_1}{\longrightarrow} \mc{EF}\right) \{ n-1 \}
\end{equation*}
where
\begin{equation*}
f_i =
\begin{cases}
x \otimes 1 - 1 \otimes y & \text{ if $i$ is odd} \\
\sum_{j=0}^{n-1} x^j \otimes y^{n-1-j} & \text{ if $i$ is even}
\end{cases}
\end{equation*}
and $x$ and $y$ represent the degree-two endomorphisms of $\mc{E}$ and $\mc{F}$ respectively. They were diagrammatically represented by a dot on an upward pointing and downward pointing strand respectively in $\mathcal{U}$ (see Section \ref{sec-cat-sl(2)}).

\begin{cor}
The equality
\begin{equation*}
[\mf{P}]=\frac{q^{n-1}(1-q^2)}{1-q^{2n}}EF=\frac{1}{[n]}EF.
\end{equation*}
holds in the usual Grothendieck group of $\mathcal{U}$. \hfill$\square$
\end{cor}

\subsection{Categorification at a prime root of unity.}
In the presence of $\partial$ we occasionally expect that the categorified Jones-Wenzl projector is quasi-isomorphic to a finite complex of $(A,A)$-bimodules. On the decategorified level this expectation comes from the observation that for certain values of $p$ the quantity $\frac{1}{[n]}$ is actually quite simple. 
For instance, let us write $n-1=kp+r$ with $0\leq r \leq p-1$. If $p$ divides $n-1$, then we have $[n]=(-1)^k$.
More generally, if $p$ does not divde $n-1$, $[n]$ is always a unit in $\mathbb{O}_p$. On the level of categories, we will see that, whenever $p$ divides $n-1$, one can construct a finite cell bimodule representing $\mf{P}$. Furthermore, for a fixed $n$, these are precisely all the prime characteristics for which this happens.
Indeed we obtain the following theorem:

\begin{thm}\label{thm-pdgjw}
Let $n$ be a positive integer, and denote by $A=A_n^!$ and $C=C_n$ the $p$-DG algebras which depend on $n$. 
\begin{enumerate}
\item[(i)] If $n-1$ is divisible by $p$, then the $p$-DG bimodule over $(A,A)$ representing $\mf{P}$ is quasi-isomorphic to a finite cell bimodule.
\item[(ii)] If $n-1$ is not divisible by $p$, then $\mf{P}$ can not be represented by a finite cell $p$-DG module over $(A,A)$.
\end{enumerate}
\end{thm}

\begin{proof}[Proof of Part (i)]
To prove the first statement we may explicitly construct a resolution of $\mf{P}$.
In this case, $n=kp+1$ for some $k\geq 0$, and the algebra $C=\Bbbk[x]/(x^{kp+1})$ is quasi-isomorphic to $ \Bbbk$
since
\begin{equation*}
\Bbbk(x \rightarrow x^2 \rightarrow \cdots \rightarrow x^{p}) \oplus
\Bbbk(x^{p+1} \rightarrow x^{p+2} \rightarrow \cdots \rightarrow x^{2p}) \oplus \cdots \oplus
\Bbbk(x^{(k-1)p+1} \rightarrow x^{(k-1)p+2} \rightarrow \cdots \rightarrow x^{kp})
\end{equation*}
is an acyclic ideal inside $C$. Thus, as a $p$-DG bimodule over $(C,A)$, we have a surjective quasi-isomorphism coming from the multiplication map of $C$ on the left module $(n)A$:
\[
C\otimes (n)A \stackrel{m}{\lra}  (n)A.
\]
It follows that the functor $\mf{P}$ is given by tensor product with the $p$-DG bimodule
\[
A(n)\otimes_C^{\mathbf{L}}(n)A \cong A(n)\otimes_C (C\otimes (n)A)\cong A(n)\otimes (n)A.
\]
Therefore
\begin{equation*}
\mf{P} \cong \left(A(n) \otimes (n)A \right) \otimes_A^{\mathbf{L}}(-) =\mc{EF}.
\end{equation*}

\end{proof}

It is a bit more difficult to prove the second statement in the Theorem. To do so we will make some preparations.

\begin{lem}\label{lemma-unbounded-complexes}Let $A$ be a $p$-DG algebra. Suppose
\[
\xymatrix{\cdots \ar[r]^-{\phi_{m+4}} & P_{m+3} \ar[r]^-{\phi_{m+3}} & P_{m+2}\ar[r]^-{\phi_{m+2}} & P_{m+1} \ar[r]^-{\phi_{m+1}}& P_m \ar[r]& 0}
\]
is a bounded-from-above exact sequence of $p$-DG modules over $A$. Then the filtered $p$-DG module below, with every other term repeated $(p-1)$-times, is acyclic:
\[
\xymatrix{\cdots \ar[r]^-{\phi_{m+4}} & P_{m+3} \ar@{=}[r]& \cdots \ar@{=}[r] & P_{m+3}\ar[r]^-{\phi_{m+3}} \ar[r] &
P_{m+2}\ar `r[rd] `_l `[llld] _-{\phi_{m+2}} `[d] [lld]
& \\
& & P_{m+1}\ar@{=}[r]&\cdots \ar@{=}[r]& P_{m+1} \ar[r]^-{\phi_{m+1}}& P_m \ar[r]& 0.}
\]
\end{lem}
\begin{proof}
The result reduces easily to the corresponding result for short exact sequences of $p$-DG modules. So let us assume that $P_{k}=0$ when $k\geq m+3$. We will analyze the $p$-complex structure of the filtered $p$-DG module obtained.

As a $p$-complex, it is easy to see that the filtered module\footnote{Here the filtration on the $p$-DG module is the natural one with respect to the differential: $P_m$ is a sub $p$-DG module that is the starting term, and the higher terms in the filtration are obtained by adding more terms to the left of $P_m$.} has as a subcomplex 
\[
0 \lra \underbrace{\mathrm{Im}(\phi_{m+1})= \cdots = \mathrm{Im}(\phi_{m+1})}_{(p-1)~\textrm{terms}}\stackrel{\cong}{\lra} P_m \lra 0.
\]
Modulo the subcomplex, the quotient is isomorphic to
\[
0 \lra P_{m+2}\stackrel{\cong}{\lra} \underbrace{\mathrm{Ker}(\phi_{m+1})= \cdots = \mathrm{Ker}(\phi_{m+1})}_{(p-1)~\textrm{terms}}\lra 0.
\]
Both complexes are obviously acyclic. The result follows.
\end{proof}

\begin{proof}[Proof of part (ii)] To prove the statement, we will show that $\mf{P}(L_n)$ has infinite $p$-cohomology.
Since $(n)A \otimes_A L_n \cong \Bbbk$,
we have $ \mf{P}(L_n) \cong A(n) \otimes_C^{\mathbf{L}} \Bbbk$.
We are reduced to finding a property-(P) (Defnition~\ref{def-finite-cell}) resolution for $\Bbbk$ as a left $p$-DG module over $C$.

Now we start constructing the resolution. The procedure is an augmented version of the construction of a bar resolution in \cite[Theorem 6.6]{QYHopf}, which will be done inductively.
 
 Consider the short exact sequence of $p$-DG modules over $C$
\begin{equation}\label{eqn-step-1}
0\lra I \lra C \lra  \Bbbk \lra 0.
\end{equation}
Here $I$ stands for the ideal generated by $x$, which is a $p$-DG submodule inside $C$.

The surjective multiplication map $C\otimes I \lra I$ is a map of $p$-DG modules. Denote its kernel by $I_2$, and it is easy to see that
\[
I_2= \{y\otimes z-1\otimes yz|y,~z\in I\}.
\]
As $p$-complexes, $C\otimes I \cong \Bbbk\otimes I \oplus I \otimes I$, and thus $I_2\cong I\otimes I$. Gluing this map to to the above short exact sequence, we get an exact sequence of left $p$-DG modules
\begin{equation}\label{eqn-step-2}
0 \lra I_2 \lra C\otimes I \lra C \lra \Bbbk \lra 0.
\end{equation}

Inductively, suppose we have constructed an exact sequence of $p$-DG modules
\begin{equation}\label{eqn-induction-step}
0 \lra I_{k-1} \lra C\otimes I_{k-2} \lra \cdots \lra C\otimes I \lra C \lra \Bbbk \lra 0
\end{equation}
where the arrow between each $C\otimes I_r\lra C\otimes I_{r-1}$ ($r=1,\dots, k-2$) is given by the multiplication of $C$ on a $p$-DG module; the kernel $I_{r+1}=\{y\otimes w -1\otimes y w|y\in I,~w \in I_{r}\}$ of the arrow, as a $p$-complex, is isomorphic to $I^{\otimes (r+1)}$. 

Then glue to the left most end of the sequence above the $C$ algebra action map
\begin{equation}\label{eqn-kernel-elts}
C\otimes I_{k-1}\lra I_{k-1}.
\end{equation}
by the discussion similar to the $k=2$ case, one sees that the kernel of the map is a $p$-DG submodule inside $C\otimes I_{k-1}$, which can be identified with
\[
I_{k}=\{y\otimes w -1\otimes y w |y\in I,~w \in I_{k-1}\},
\]
whose underlying $p$-complex is isomorphic to $I\otimes I_{k-1}\cong I^{\otimes k}$. 
This finishes the induction step, and we obtain a bounded-above sequence of $p$-DG modules ($I_1:=I$).
\begin{equation}
\label{eqn-step-infty}
\cdots \lra C\otimes I_k \lra C\otimes I_{k-1} \lra \cdots \lra C\otimes I_1 \lra C \lra \Bbbk\lra 0. 
\end{equation}

Now, applying Lemma \ref{lemma-unbounded-complexes} to the sequence \eqref{eqn-step-infty}, we obtain a desired filtered module which satisfies property (P) (Definition \ref{def-finite-cell}) and is quasi-isomorphic to $\Bbbk$:
\begin{equation}\label{eqn-filtered-property-P-mod}
\begin{gathered}
\xymatrix{\cdots \ar[r]^-{m} & C\otimes I_3 \ar@{=}[r]& \cdots \ar@{=}[r] & C\otimes I_3\ar[r]^-{m} \ar[r] &
C\otimes I_2 \ar `r[rd] `_l `[llld] _-{m} `[d] [lld]
& \\
& & C\otimes I_1 \ar@{=}[r]&\cdots \ar@{=}[r]& C\otimes I_1 \ar[r]^-{m}& C \ar[r]& 0,}
\end{gathered}
\end{equation}
where the maps $m$ indicate the algebra $C$ multiplication acting on a $C$-module.

It then follows that $\mf{P}(L_n)$ is quasi-isomorphic to the $p$-complex
\[
\xymatrix{\cdots \ar[r]^-{m} & A(n)\otimes I_3 \ar@{=}[r]& \cdots \ar@{=}[r] & A(n)\otimes I_3\ar[r]^-{m} \ar[r] &
A(n)\otimes I_2 \ar `r[rd] `_l `[llld] _-{m} `[d] [lld]
& \\
& & A(n)\otimes I_1 \ar@{=}[r]&\cdots \ar@{=}[r]& A(n)\otimes I_1 \ar[r]^-{m}& A(n) \ar[r]& 0.}
\]
As the summand decomposition \eqref{eqn-C-summand-of-A(n)} is also a decomposition of $p$-DG modules, we have that $A(n)$ contains $x^{n-1}C\cong \Bbbk \{2n-2 \}$ as a $p$-complex direct summand. Hence $\mf{P}(L_n)$ contains the infinite $p$-complex (up to an overall shift of $2n-2$)
\[
\xymatrix{\cdots \ar[r]^-{0} & I^{\otimes 3} \ar@{=}[r]& \cdots \ar@{=}[r] & I^{\otimes 3}\ar[r]^-{0} \ar[r] &
I^{\otimes 2} \ar `r[rd] `_l `[llld] _-{0} `[d] [lld]
& \\
& & I \ar@{=}[r]&\cdots \ar@{=}[r]& I \ar[r]^-{0}& \Bbbk \ar[r]& 0,}
\]
where the zero-maps come from the fact that the element $x^{n-1}$ kills the ideal $I$: Under the map $m$, we have, for any $y\otimes w -1\otimes y w \in I_k$, where $y\in I$ and $w \in I_{k-1}$,
\[
x^{n-1}\otimes_C (y\otimes w -1\otimes y w)=-x^{n-1}\otimes y w \mapsto 
-x^{n-1}y w=0.
\] 
The theorem now follows.
\end{proof}

\subsection{Two examples.} In Theorem ~\ref{thm-pdgjw}, when $n-1$ is not divisible by $p$, the number of copies of $C$ in each piece of the property-(P) replacement of $\Bbbk$ grows exponentially. However, there are two special cases where the bar resolution used can be replaced by a much smaller resolution: the number of copies of $C$ in each degree is almost a constant.

\begin{example}
We consider the case when $n=(k+1)p$ for some $k\in \N$. As before, first we take $C_0:=C\cdot 1_0$ and map it onto $\Bbbk$ by identifying $\Bbbk$ with the lowest degree subspace spanned by $1_0$. Next, we form the filtered complex (c.f.~Example~\ref{eg-finite-cell})
\[
\xymatrix{C1_1\ar[r]^-{x^{kp+1}} & C1_0}.
\]
Since $\dif^{p-1}(x^{kp+1})=0$, this is a rank-two finite cell module. 

We compute that
\begin{align*}
\dif^{p-1}(x^{kp+1}1_1)& =
\sum_{i=0}^{p-1}{p-1 \choose i}\dif^i(x^{kp+1})\dif^{p-i-1}(1_1)\\
& =\sum_{i=0}^{p-2}\dfrac{(p-1)!}{i!(p-1-i)!}i! x^{kp+1+i}\dif^{p-i-2}(x1_0)\\
& = \sum_{i=0}^{p-2}\dfrac{(p-1)!}{(p-1-i)}x^{(k+1)p}1_0.
\end{align*}
The second equality holds since $\dif^{p-1}(x^{kp+1})=0$. Now, as $i$ ranges over $\{0,\dots, p-2\}$, $1/(p-1-i)$ ranges over all non-zero elements of $\F_p$, therefore the sum in the last term is zero.

It follows that we may repeat the construction, and we obtain an infinite filtered module satisfying property (P) as follows:
\[
\xymatrix{
\cdots\ar[r]^-{x^{kp+1}} & C\{2(r+1)kp\} \ar[r]^-{x^{kp+1}} & C\{2rkp\}\ar[r]^-{x^{kp+1}} & \cdots \ar[r]^-{x^{kp+1}} & C\{2kp\}\ar[r]^-{x^{kp+1}} & C .
}
\]
This resolution of $\Bbbk$ as a left $p$-DG module has an interesting implication in the Grothendieck group $K_0(C)$, which says that
\[
[\Bbbk]= \sum_{r=0}^{\infty} q^{2rkp}[C].
\]
Since $q^{2p}=1$ in $\mathbb{O}_p$, this equation reduces to
\begin{equation}\label{eqn-symbol-identity}
[\Bbbk] = \sum_{r=0}^{\infty} [C].
\end{equation}
Although the identity $\sum_{r=0}^\infty 1 =1/(1-1)$ no longer makes sense, yet equation \eqref{eqn-symbol-identity} can be understood as
\[
[C]= [\Bbbk]+[I]=[\Bbbk]+\sum_{i=1}^{p-1}q^{2i}[\Bbbk]=[\Bbbk]-[\Bbbk]=0.
\]
\end{example}

\begin{example}
In this example we consider Theorem~\ref{thm-pdgjw} when $n=kp+2$ ($k\geq 0$), in which case the $p$-DG resolution of $\Bbbk$ over $C$ may also be simplified significantly.

For $j=0,\ldots,k-1$ there are contractible $p$-complexes of $C$ of the form
\begin{equation*}
(x^{jp+1} \rightarrow x^{jp+2} \rightarrow \cdots \rightarrow x^{(j+1)p}).
\end{equation*}
Contracting these $p$-complexes one establishes a quasi-isomorphism between $C$ and $\Bbbk[x^{kp+1}]/(x^{kp+1})^2$ with trivial differential.  

As a $p$-DG module over $C$, the trivial module $\Bbbk$ is quasi-isomorphic to

\begin{equation}
\begin{gathered}\label{Cresr=p-1}
\xymatrix{\cdots \ar[r]^-{x^{kp+1}} & C \{ 2(3kp-2p+4)   \} \ar@{=}[r]& \cdots \ar@{=}[r] & C \{ 2(3kp-p+2) \} \ar[r]^-{x^{kp+1}} \ar[r] &
C \{ 2(3kp-p+2) \} \ar `r[rd] `_l `[lllld] _-{ x^{kp+1}} `[d] [llld]
& \\
&  C \{ 2(kp-p+2)  \} \ar@{=}[r]&\cdots \ar@{=}[r]& C \{ 2kp \} \ar[r]^-{ x^{kp+1} }& C \ar[r]& 0}
\end{gathered}
\end{equation}
Similar as above, this equation has an interesting Grothendieck group implication saying that
\[
[\Bbbk]=\sum_{r=0}^{\infty}(-1)^r q^{2r}[C].
\]
Although, as before, $q^2$ is a root of unity in $\mathbb{O}_p$, and it does not make sense for $\sum_{r=0}^{\infty}(-1)^r q^{2r}$ to converge, this equation can, however, be ``explained'' via the equality of symbols
\[
[C]=[\Bbbk] +[I] =[\Bbbk]+q^{2(kp+1)}[\Bbbk]=(1+q^2)[\Bbbk].
\]

%\begin{equation}
%\cdots {\lra} C  \stackrel{\cdot x^{kp+1}}{\lra}
%\cdots {\lra}
%\underbrace{C \langle 2(3kp-2p+4)   \rangle  =\cdots =C \langle $2(3kp-p+2) \rangle }_{p-1}  \stackrel{x^{kp+1}}{\lra}
%C \langle 2(kp-p+2)   \rangle 
%\stackrel{\cdot x^{kp+1}}{\lra}\underbrace{C \langle 2(kp-p+2)  %\rangle =\cdots =C \langle 2kp \rangle}_{p-1}  \stackrel{\cdot x^{kp%+1}}{\lra} C.
%\end{equation}
%Tensoring the module in ~\eqref{Cresr=p-1} by $A(n)$ over $C$, we obtain a cofibrant resolution of $\Bbbk$ as an $A$-module
%\begin{equation}\label{Aresr=p-1}
%\begin{gathered}
%\xymatrix{\cdots \ar[r]^-{\phi} & A \{ 2(3kp-2p+4)   \} \ar@{=}[r]& \cdots \ar@{=}[r] & A \{ 2(3kp-p+2) \} \ar[r]^-{\phi} \ar[r] &
%A \{ 2(3kp-p+2) \} \ar `r[rd] `_l `[lllld] _-{ \phi } `[d] [llld]
%& \\
%&  A \{ 2(kp-p+2)  \} \ar@{=}[r]&\cdots \ar@{=}[r]& A \{ 2kp \} \ar[r]^-{ \phi }& A \ar[r]& 0,}
%\end{gathered}
%\end{equation}
%\begin{equation*}
%\cdots {\lra} A(n)  \stackrel{\cdot (n|n-1|n)^{kp+1}}{\lra}
%\cdots {\lra}
%\underbrace{A(n)\{ 2(3kp-2p+4)\}=\cdots =A(n) \{ 2(3kp-p+2) \}}%_{p-1}  \stackrel{(n|n-1|n)^{kp+1}}{\lra}
%A(n)\{ 2(2kp-p+2)   \}  \stackrel{\cdot (n|n-1|n)^{kp+1}}{\lra}%\underbrace{A(n)\{ 2(kp-p+2)  \}=\cdots =A(n) \{ 2kp \}}_{p-1}  %\stackrel{\cdot (n|n-1|n)^{kp+1}}{\lra} A(n).
%\end{equation*}
%where $\phi$ is, up to a grading shift, the left $A$-module map of attaching $(n|n-1|n)^{kp+1}$ to an element of $A(n)$ on the right.
\end{example}

Both examples are instances of why one wants to restrict to \emph{compact objects} for defining the Grothendieck group of the triangulated category $\mc{D}(C)$ of a $p$-DG algebra. The module $\Bbbk$ in these cases are not compact in $\mc{D}(C)$. We refer the reader to \cite{EQ1} for a discussion about compact $p$-DG modules over a \emph{positive $p$-DG algebra}. This notion includes both $C$ and $A$ in this paper as special instances.

\vspace{0.2in}

\bibliographystyle{amsalpha}
\bibliography{qy-bib}

\end{document}